\documentclass[11pt,reqno]{amsart}
\usepackage{amsmath,amsfonts,amssymb,amscd,amsthm,amsbsy,bbm, epsf,calc, enumerate}
\usepackage{color, esint}
\usepackage[usenames,dvipsnames]{xcolor}
\usepackage{datetime}

\usepackage{thmtools, thm-restate}
\usepackage{thm-patch, thm-kv, thm-autoref}
\usepackage[colorlinks=true, urlcolor=blue, linkcolor=blue, citecolor=black]{hyperref}

\usepackage[top=1.30in, bottom=1.15in, left=1.07in, right=1in]{geometry}

\numberwithin{equation}{section}
\newcounter{hours}\newcounter{minutes}

\theoremstyle{plain}
\declaretheorem[title=Theorem, parent=section]{thm}
\declaretheorem[title=Lemma,sibling=thm]{lem}

\declaretheorem[title=Proposition,sibling=thm]{prop}
\declaretheorem[title=Definition,sibling=thm]{DEF}
\declaretheorem[title=Remark,sibling=thm]{rem}

\def\G{\mathcal G}
\def\I{\mathcal I}
\def\M{{\mathcal M}}

\def\real{{\mathbb R}}

\def\Indicator{{\mathbbm{1}}}

\def\ep{\varepsilon}
\def\al{\alpha}
\def\del{\delta}
\def\om{\omega}
\def\Om{\Omega}
\def\gam{\gamma} 
\def\lam{\lambda}
\def\Lam{\Lambda}

\def\sig{\sigma}
\def\Div{\textnormal{div}}
\def\grad{\nabla}
\def\diam{\textnormal{diam}}

\def\Tr{\textnormal{tr}}
\def\Id{\textnormal{Id}}

\DeclareMathOperator{\Hess}{Hess}
\def\intersect{\cap}

\DeclareMathOperator{\spt}{spt}

\newcommand{\abs}[1]{\left| #1 \right|}

\newcommand{\norm}[1]{\lVert#1\rVert}

\DeclareMathOperator{\vol}{vol}

\begin{document}
	
%VERSION NOTES: this is version 6.  it is submitted V1

\title{Estimates for Dirichlet-to-Neumann maps as integro-differential operators}

\author{Nestor Guillen}
\author{Jun Kitagawa}
\author{Russell W. Schwab}

\address{Department of Mathematics\\
University of Massachusetts, Amherst\\
Amherst, MA  90095}
\email{nguillen@math.umass.edu}

\address{Department of Mathematics\\
Michigan State University\\
619 Red Cedar Road \\
East Lansing, MI 48824}
\email{kitagawa@math.msu.edu}

\address{Department of Mathematics\\
Michigan State University\\
619 Red Cedar Road \\
East Lansing, MI 48824}
\email{rschwab@math.msu.edu}

\begin{abstract}
	
	Some linear integro-differential operators have old and classical representations as the Dirichlet-to-Neumann operators for linear elliptic equations, such as the 1/2-Laplacian or the generator of the boundary process of a reflected diffusion.  In this work, we make some extensions of this theory to the case of a \emph{nonlinear} Dirichlet-to-Neumann mapping that is constructed using a solution to a \emph{fully nonlinear} elliptic equation in a given domain, mapping Dirichlet data to its normal derivative of the resulting solution.  Here we begin the process of giving detailed information about the L\'evy measures that will result from the integro-differential representation of the Dirichlet-to-Neumann mapping.  We provide new results about both linear and nonlinear Dirichlet-to-Neumann mappings.  Information about the L\'evy measures is important if one hopes to use recent advancements of the integro-differential theory to study problems involving Dirichlet-to-Neumann mappings.

\end{abstract}

\date{\today,\  ArXiv version 1}
\thanks{The authors all acknowledge partial support from the NSF leading to the completion of this work: N. Guillen DMS-1201413 and DMS-1700307; J. Kitagawa DMS-1700094; R. Schwab DMS-1665285.  They would like to thank Rodrigo Ba\~nuelos and Renming Song for helpful information on background results appearing in Section \ref{sec:Tools}.}

\maketitle
 \markboth{N. Guillen, J. Kitagawa, R. Schwab}{D-to-N and Integro-Differential Operators}
%%%%%%%%%%%%%%%%%%%%%%%%%%%%%%%%%%%%%%%%%%%%%%
%%%%%%%%%%%%%%%%%%%%%%%%%%%%%%%%%%%%%%%%%%%%%%

%%%%%%%%%%%%%%%%%%%%%%%%%%%%%%%%%%%%%%%%%%%%%%
%%%%%%%%%%%%%%%%%%%%%%%%%%%%%%%%%%%%%%%%%%%%%%
%%%%%%                   %%%%%%%%%%%%%%%%%%%%%
%%%%%%    INTRO          %%%%%%%%%%%%%%%%%%%%%
%%%%%%                   %%%%%%%%%%%%%%%%%%%%%
%%%%%%%%%%%%%%%%%%%%%%%%%%%%%%%%%%%%%%%%%%%%%%
%%%%%%%%%%%%%%%%%%%%%%%%%%%%%%%%%%%%%%%%%%%%%%
%%%%%%%%%%%%%%%%%%%%%%%%%%%%%%%%%%%%%%%%%%%%%%
%%%%%%%%%%%%%%%%%%%%%%%%%%%%%%%%%%%%%%%%%%%%%%
%%%%%%%%%%%%%%%%%%%%%%%%%%%%%%%%%%%%%%%%%%%%%%

\section{Introduction, Assumptions, Background}\label{sec:FirstSection}
\setcounter{equation}{0}

\subsection{Introduction}\label{sec:Intro}

In this work, we explore the precise connection between integro-differential operators acting on functions in, e.g. $C^{1,\al}(\partial\Om)$, for $\Om$ a nice domain in $\real^{n+1}$, and operators that are the Dirichlet-to-Neumann mappings (from now on, ``D-to-N'')  for various elliptic equations in $\Om$.  We prove estimates on the L\'evy measures (explained below) that appear in the integro-differential representation of these D-to-N operators.  Our motivating interest is the D-to-N for fully nonlinear elliptic equations (itself, a nonlinear mapping), and the resulting integro-differential theory.  However, in the course of exploring the nonlinear setting, we noticed the linear theory seems not to be recorded in any place, except for the case of the Laplacian, where Hsu \cite[Section 4]{Hsu-1986ExcursionsReflectingBM} gave a complete description for the boundary process of a reflected Brownian motion in a smooth domain.  In that sense, this paper can be considered an extension of \cite{Hsu-1986ExcursionsReflectingBM} to the case of more general linear and nonlinear equations.

The set-up for the D-to-N is as follows.  Let $\Om$ be a bounded domain (assumed throughout for simplicity, but many adaptations to unbounded domains are possible),  let $\phi\in C^{1,\al}(\partial\Om)$, and generically, we take $U_\phi$ as the unique solution of
\begin{align}\label{eqIn:BulkExtensionGeneric}
	\begin{cases}
		F(U_\phi,x)=0\ &\text{in}\ \Om\\
		U_\phi=\phi\ &\text{on}\ \partial\Om.
	\end{cases}
\end{align} 
Here $F$ may be any one of the possible operators:

\begin{align}
	F(U,x) &= \Div(A(x)\grad U),\  \text{with}\ A\in C^\al(\Om)\ \text{and uniformly elliptic}, \label{eqIn:BulkFDiv}\\
	F(U,x) &= \Tr(A(x)D^2U),\ \text{with}\ A\in C^\al(\Om)\ \text{and uniformly elliptic}, \label{eqIn:BulkFNondiv}\\
	F(U,x) &= F(D^2U,x),\ \text{with}\  F\  \text{uniformly elliptic with (locally) H\"older coefficients}\label{eqIn:BulkFNonlinear}.
\end{align}
The precise assumptions appear in more detail below.  The D-to-N, which we call $\I$, is defined as
\begin{align}\label{eqIn:DefOfDtoN}
	\phi\mapsto \partial_\nu U_\phi, \text{denoted as}\ \I(\phi,x):= \partial_\nu U_\phi(x),
\end{align}
where $\nu(x)$ is the \emph{inward} normal vector to $\partial\Om$ at $x$.  In each of these three situations, it is not hard to check (which we do below)  that the D-to-N, is not only well defined as a map from $C^{1,\al}(\partial\Om)$ to $C^\al(\partial\Om)$, but it also enjoys what we call the global comparison property (defined below, Definition \ref{def:GCP}).  This is the simple fact that the operator, $\I$, preserves ordering between any two functions that are globally ordered on $\partial\Om$ and agree at a point in their domain.  The global comparison property of these D-to-N operators is the driving feature behind our results.

In the first two of the cases listed in (\ref{eqIn:BulkFDiv}) and (\ref{eqIn:BulkFNondiv}), $F$, and hence also $\I$ are linear operators.  It was proved in the 1960's, by Bony-Courr\`ege-Priouret \cite{BonyCourregePriouret-1966FellerSemiGroupOnDifferentialVariety}, through linearity and the global comparison property, that $\I$ must be an integro-differential operator of the form
\begin{align}\label{eqIn:IntDiffLinear}
	\I(\phi,x) = b(x)\cdot\grad \phi(x) + \textnormal{p.v.}\ \int_{\partial\Om} (\phi(h)-\phi(x)) \mu(x,dh),
\end{align}
for some tangential vector field, $b$, and a L\'evy measure, $\mu(x,\cdot)$.  Recently, two of the authors, in \cite{GuSc-2016MinMaxNonlocalarXiv}, obtained a min-max representation for nonlocal and nonlinear operators that results in a formula similar to (\ref{eqIn:IntDiffLinear}), and in one of our theorems below, we invoke this result to show that  $\I$ in the nonlinear setting will be a min-max over a family of linear operators of the form (\ref{eqIn:IntDiffLinear}). We will record this result precisely in our main results, listed below.  We note to the reader that we have collected various notations in Section \ref{sec:Notation}.

Our goal is not to re-derive (\ref{eqIn:IntDiffLinear}), but rather to more precisely detail the properties of $b$ and $\mu$. In order to connect $\I$ to the recent activity in the theory of linear and nonlinear integro-differential equations and to exploit some recent results,  further properties of the L\'evy measures ($\mu$ in (\ref{eqIn:IntDiffLinear})) are required to know which integro-differential results are applicable.  This is the main goal of the article, and our main results are as follows.  We note that we have separated many of the assertions for the sake of presentation and that they hold under different assumptions on the regularity of $\partial\Om$.  Theorems \ref{thm:MainLinear} and \ref{thm:HolderXDependenceTVNorm} have somewhat standard assumptions on $\partial\Om$, and Theorem \ref{thm:HolderDrift} requires significantly more regularity of $\partial\Om$.

In the following results, $\partial\Om$ will be viewed as a Riemannian manifold whose Riemannian metric is induced by the Euclidean inner product on $\real^{n+1}$.
 
\begin{thm}[Linear D-to-N]\label{thm:MainLinear}
	Assume that $F$ is as in one of (\ref{eqIn:BulkFDiv}) or (\ref{eqIn:BulkFNondiv}).
	If $\Om\subset\real^{n+1}$ is bounded and $\partial\Om$ is of class $C^{3}$ with an injectivity radius bounded from below by $2r_0>0$, and $\I$ is defined via (\ref{eqIn:BulkExtensionGeneric}), (\ref{eqIn:DefOfDtoN}), then there exists a vector field, $b$, and a family of measures parametrized by $x$, $\mu(x,dh)$, such that for all $\phi\in C^{1,\al}(\partial\Om)$
	\begin{align}
		\I(\phi,x) = (b(x),\grad \phi(x))_g + 
		\int_{\partial\Om}\left( \phi(h)-\phi(x)-
		\Indicator_{B_{r_0}(x)}(h)(\grad \phi(x), \exp^{-1}_x(h))_g \right)
		\mu(x,dh).\label{eqIn:Thm1DtoNIntDiffLinear}
	\end{align}

	\noindent
	Furthermore, $b$ and $\mu$ satisfy:
	\begin{enumerate}[(i)]
		\item For all $x\in\partial\Om$, $\mu(x,\cdot)$ has a density, $\mu(x,dh)=K(x,h)\sig(dh)$,
		\item There exist universal $c_1>0$ and $c_2\geq c_1$ so that for all $x\in\partial\Om$, $h\in\partial\Om$, and $x\not=h$, $\displaystyle c_1 d(x,h)^{-n-1}\leq K(x,h)\leq c_2 d(x,h)^{-n-1}$,
		\item $b$ is bounded.
	\end{enumerate}
	We note that $c_1$, $c_2$, and the bound for $b$ depend only on the $C^{1,\al}$ nature of $\partial\Om$ in the case of $F$ in (\ref{eqIn:BulkFDiv}) and only on the $C^{2}$ nature of $F$  for (\ref{eqIn:BulkFNondiv}). 
\end{thm}

In addition, if we assume more regularity of $\partial\Om$, one can obtain more information about the constituents of the representation in (\ref{eqIn:Thm1DtoNIntDiffLinear}).

\begin{thm}[H\"older Drift]\label{thm:HolderDrift}
	If additionally for $\Om$ as above, it is assumed that $\partial\Om$ is of class $C^5$, then $b$ as in (\ref{eqIn:Thm1DtoNIntDiffLinear}) is H\"older continuous in $x$.
\end{thm}

\begin{rem}
	We note that for Theorem \ref{thm:HolderDrift}, we openly admit that assuming $\partial\Om$ is $C^5$ is most likely more than necessary.  However, given that our eventual interest is the hope that some Krylov-Safonov type theorems will be developed for the resulting integro-differential operators, the regularity of $b$ is a low priority.  In the context of Krylov-Safonov results, it is the boundedness of $b$ that is more important, e.g. akin to the results in \cite{Silv-2011DifferentiabilityCriticalHJ}.
\end{rem}

Our next result shows that the L\'evy measures (away from the singularity) are H\"older continuous in the TV norm.
Specifically, it shows that the L\'evy measure in (\ref{eqIn:Thm1DtoNIntDiffLinear}), restricted to the set outside of a small ball at the singularity, when $h=x$, enjoys a control that depends on the size of the ball as well as a H\"older fashion in $x$.

We denote by $\textbf{M}(\partial\Om)$ the space of signed measures on $\partial \Om$, and by $\norm{\cdot}_{TV}$ the total variation norm of a signed measure on $\partial \Om$. Recall that (see \cite[Section 29]{HalmosBook50})
\begin{align}\label{eqn:TVcharacterization}
 \norm{\mu}_{TV}=\sup\left\{\abs{\int_{\partial \Om}\phi(h)\mu(dh)}\ :\  \phi\in L^\infty(\partial \Om),\ \norm{\phi}_{L^\infty(\partial \Om)}\leq 1\right\}.
\end{align}

\begin{thm}[H\"older in TV Norm]\label{thm:HolderXDependenceTVNorm}
	    For a fixed $\delta>0$, define $\mu_\del: \partial \Om\to \textbf{M}(\partial \Om)$ by  
	  \begin{align*}
	   \mu_\del(x):=\chi_{\partial \Om \setminus B_\delta(x)}(\cdot)\mu(x, \cdot).
	  \end{align*}
	  Then there exists an $\alpha\in (0, 1)$ such that for $\delta>0$ sufficiently small, 
	  \begin{align*}
	  \mu_\del\in C^{\alpha}_{loc}\left(\partial \Om; \left(\textbf{M}(\partial \Om), \norm{\cdot}_{TV}\right) \right).
	  \end{align*}
	  More specifically, for each $\delta$ there exists a constant $C>0$ such that
	  for any $x_0\in \partial \Om$ and $x_1$, $x_2 \in B_{\delta/4}(x_0)$ it holds that
	  \begin{align*}
	   \norm{\mu_\del(x_1)-\mu_\del(x_2)}_{TV}\leq \frac{C}{\delta^2} d(x_1, x_2)^\alpha.
	  \end{align*}
	  Here $C$ depends on universal parameters and the lower bound on the Ricci curvature of $\partial\Om$, $\alpha$ arises from the $C^{1,\alpha}$ and $C^2$ character of $\partial\Om$ for $F$ respectively in (\ref{eqIn:BulkFDiv}) and (\ref{eqIn:BulkFNondiv}), while the smallness required of $\delta$ depends only on $\partial \Om$.
\end{thm}

Next, we have the result for the nonlinear version of the D-to-N mapping.

\begin{thm}[Nonlinear D-to-N]\label{thm:MainNonlinear}
	If $\Om$ is bounded and $\partial\Om$ is of class $C^{3}$ with an injectivity radius bounded from below by $r_0>0$, and $\I$ is defined via (\ref{eqIn:BulkExtensionGeneric}), (\ref{eqIn:DefOfDtoN}), using $F$ as in (\ref{eqIn:BulkFNonlinear}), then $\I$ is a min-max over an appropriate family of operators given by $b^{ij}$ and $\mu^{ij}$,
	\begin{align}\label{eqIn:NonlinearMinMax}
		\I(\phi,x) = &\min_i \max_j \{ f^{ij}(x) + c^{ij}(x)\phi(x) +   (b^{ij}(x),\grad \phi(x))_g \nonumber \\
		&\ \ \ \  + 
		\int_{\partial\Om}\left( \phi(h)-\phi(x)-
		\Indicator_{B_{r_0}(x)}(h)(\grad \phi(x), \exp^{-1}_x(h))_g \right)
		\mu^{ij}(x,dh)
		\}.
	\end{align}
	Furthermore, 
	\begin{enumerate}[(i)]
		\item the L\'evy measures satisfy, uniformly in $i,j$, for $x,h\in\partial\Om$, $x\not=h$,
		\begin{enumerate}[(a)]
			\item a ring estimate: there exist universal $R$, $C_1$, $C_2$, all $>0$,  so that for all $0<r\leq R$
			$\displaystyle C_1 r^{-1}\leq\mu^{ij}(x, B_{2r}(x)\setminus B_r(x))\leq C_2 r^{-1}$
			\item lower bound: there exists universal $R>0$ and $\eta>0$ so that for all $h$ with $d(x,h)<R$ and $\displaystyle 0<r<\frac{d(x,h)}{10}$\\ 
			$\displaystyle \frac{C_1r^{\eta}}{(d(x,h))^{\eta+1}}\leq \mu^{ij}(x,B_r(h))$.
		\end{enumerate}
		\item $b^{ij}$ is bounded uniformly in $i,j$.
	\end{enumerate}
	The constants depend on universal parameters and only on the $C^2$ nature of $\partial\Om$.
\end{thm}

\begin{rem}
	We want to point out to the reader that in both Theorems \ref{thm:MainLinear} and \ref{thm:MainNonlinear}, the \emph{existence and boundedness} of the $b$ and $\mu$ (or $f^{ij}$, $c^{ij}$, $b^{ij}$, $\mu^{ij}$ in the min-max) are \emph{not new}.  In the linear case, this is a result of Bony-Courr\`ege-Priouret \cite{BonyCourregePriouret-1966FellerSemiGroupOnDifferentialVariety}, and in the nonlinear case by two of the authors \cite{GuSc-2016MinMaxNonlocalarXiv}.  The new part of these results are the properties (i)-(ii) in Theorem \ref{thm:MainLinear} and (i) in Theorem \ref{thm:MainNonlinear}. 
\end{rem}

\begin{rem}
	The ring estimate in Theorem \ref{thm:MainNonlinear} (i-a), although not sufficient for regularity theory yet, at least shows the the L\'evy measures, $\mu^{ij}$, contain the same amount of mass on every ring, $B_{2r}(x)\setminus B_r(x)$, as does the 1/2-Laplacian.  The lower bound in (i-b) at least shows that the L\'evy measures, $\mu^{ij}$ are supported everywhere on $\partial\Om$, but that possibly they have a scaling that is other than the one for surface measure (scaling by the power $n$), and we note that one expects $\eta$ in this situation to be large (so balls may carry small mass), as opposed to the more regular situation where one has $\eta=n$.
\end{rem}

\begin{rem}[$\partial\Om\in C^3$]
	In both Theorems \ref{thm:MainLinear} and \ref{thm:MainNonlinear}, there is an assumption that $\partial\Om$ should be $C^3$.  This is a technical assumption arising from the way that the main result in \cite{GuSc-2016MinMaxNonlocalarXiv} was proved.  There, it is a technical assumption made for simplicity, and so also here it plays the same role.  The more important assumptions arise from results about boundary regularity of solutions of elliptic equations, in which case, they depend on $C^{1,\al}$ or $C^2$ ingredients, depending upon the type of equation.
\end{rem}

\begin{rem}[Boundedness of $\Om$]
	In all of our results, we have assumed that $\Om$ is bounded.  This assumption is made purely for simplicity and uniformity, and we note that in many contexts that the outcomes of all of the theorems will remain true, provided the supporting results we invoke have modifications to unbounded domains.
\end{rem}

%%%%%%%%%%%%%%%%%%%%%%%%%%%%%%%%%%%%%%%%%%%%%%
%%%%%%%%%%%%%%%%%%%%%%%%%%%%%%%%%%%%%%%%%%%%%%
%%%%%%%%%%%%%%%%%%%%%%%%%%%%%%%%%%%%%%%%%%%%%%

\subsection{Some Notation}\label{sec:Notation}

Here we collect a list of various notation used in this paper.

\begin{itemize}
	\item We will use capitalized function names, e.g. $U$ (and others), to denote functions defined in the domain, $\overline\Om$, and we will use lower case function names, e.g. $u$ (and others), to denote functions on the boundary, $\partial\Om$.  A function solving an equation with prescribed boundary data would then appear as $U_u$.
	\item $\Om$ is an open bounded domain in $\real^{n+1}$, that is connected, and with $\partial\Om$ having an injectivity radius, ${\rm inj}(\partial \Om)$,  bounded from below by $r_0>0$.
		
	\item $n$ is the dimension of $\partial\Om$, with $\Om\subset\real^{n+1}$.
	\item $\mu(x,\cdot)$ or $\mu^{ij}(x,\cdot)$ is a L\'evy measure used in the integro-differential representation of $\I$.
	\item $d(x,y)$ is the geodesic distance between $x$ and $y$ when $x,y\in\partial\Om$.
	\item $\sig$ is surface measure on $\partial\Om$
	\item $\nu(x)$ is the \emph{inward} normal vector to $\partial\Om$ at $x\in\partial\Om$.
	\item $B_r(x)\subset\partial\Om$ is a geodesic ball in $\partial\Om$ and $B^{n+1}_r(x)\subset\real^{n+1}$ is a Euclidean ball in $\real^{n+1}$.
	\item The word \textbf{universal} is used for constants that depend only on dimension, ellipticity, $\partial\Om$, and the coefficients of $F$ in (\ref{eqIn:BulkFDiv})--(\ref{eqIn:BulkFNonlinear}).
	\item $G(x,y)$ will be the Green's function for $\Om$ and a linear operator of the form (\ref{eqIn:BulkFDiv}) or (\ref{eqIn:BulkFNondiv}).
\end{itemize}

%%%%%%%%%%%%%%%%%%%%%%%%%%%%%%%%%%%%%%%%%%%%%%
%%%%%%%%%%%%%%%%%%%%%%%%%%%%%%%%%%%%%%%%%%%%%%
%%%%%%%%%%%%%%%%%%%%%%%%%%%%%%%%%%%%%%%%%%%%%%

\subsection{Some Definitions}\label{sec:Definitions}

\begin{DEF}\label{def:GCP}
	The global comparison property for $I: C^{1,\al}(X)\to C^0(X)$ requires that for all $u,v\in C^{1,\al}(X)$ such that $u(x)\leq v(x)$ for all $x\in X$ and such that for some $x_0$, $u(x_0)=v(x_0)$, then the operator $I$ satisfies $I(u,x_0)\leq I(v,x_0)$.  That is to say that $I$ preserves ordering of functions on $X$ at any points where their graphs touch.
	
\end{DEF}

\begin{DEF}\label{def:MaximalPucci}
	The second order $(\lam,\Lam)$-Pucci extremal operators are defined as $\M^-$ and $\M^+$, for a function, $U$ that is second differentiable at $x$, via
	\begin{align*}
		\M^-(U,x) = \min_{\lam\Id\leq B\leq \Lam\Id}\left( \Tr(BD^2U(x)) \right)\ \ \text{and}\ \ 
		\M^+(U,x) = \max_{\lam\Id\leq B\leq \Lam\Id}\left( \Tr(BD^2U(x)) \right).
	\end{align*}
	When $\{v_i\}_{i=1,\dots,n+1}$ are the eigenvalues of $D^2U(x)$, an equivalent representation is
	\begin{align*}
		\M^-(U,x) = \Lam\sum_{v_i\leq 0}v_i + \lam\sum_{v_i>0} v_i
		\ \ \text{and}\ \ 
		\M^+(U,x) = \lam\sum_{v_i\leq 0}v_i + \Lam\sum_{v_i>0} v_i.
	\end{align*}
\end{DEF}

\begin{DEF}
	We say that $F$ is $(\lam,\Lam)$-uniformly elliptic in the cases (\ref{eqIn:BulkFDiv}) and (\ref{eqIn:BulkFNondiv}) if
	\begin{align*}
		\lam\Id\leq A(x)\leq \Lam\Id\ \ \text{for all}\ \ x\in\Om,
	\end{align*}
	and in the case of (\ref{eqIn:BulkFNonlinear}) if for all $U,V\in C^2(\Om)$,
	\begin{align*}
		\M^-(U-V,x)\leq F(D^2U,x)-F(D^2V,x)\leq \M^+(U-V,x)\ \ \text{for all}\ \ x\in\Om.
	\end{align*}
\end{DEF}

We will also require the notion of harmonic measure associated to a linear equation; for details see \cite[Introduction]{CaffFabesMortSalsa-1981BoundBehavEllipticIUMJ} for the divergence case and \cite[Definition 5.16]{Kenig-1993PotentialThoeryNonDiv} for the non-divergence case.

\begin{DEF}\label{def:HarmonicMeasure}
	Given linear operators, $F$ (or sometimes $L$, below), as in (\ref{eqIn:BulkFDiv}) or (\ref{eqIn:BulkFNondiv}), it is well known that when $\phi\in C(\partial\Om)$ is prescribed, there exists a unique $U_\phi\in C(\overline \Om)$ that solves (\ref{eqIn:BulkExtensionGeneric}).  Thus, for a fixed $x$, the mapping $x\mapsto U_\phi(x)$ is well defined, and thanks to the comparison principle for these equations, is a non-negative linear functional on $C(\partial\Om)$.  We take, as a definition, that for $x$ fixed, the unique Borel measure that represents this functional to be called the $F$-Harmonic measure (or the $L$-Harmonic measure), and we denote this measure as $\om_x$.  That is to say, $\om_x$, is uniquely characterized by 
	\[
	\forall\ \phi\in C(\partial\Om),\ U_\phi(x)=\int_{\partial\Om} \phi(y)\om_x(dy).
	\]
\end{DEF}

\begin{DEF}\label{def:GreensFun}
	Given linear operators, $F$ (or sometimes $L$, below), as in (\ref{eqIn:BulkFDiv}) or (\ref{eqIn:BulkFNondiv}), the Green's function (see e.g. \cite[Section 2]{Kenig-1993PotentialThoeryNonDiv} or \cite{LittStamWein-1963RegularPointsEllipticEqAnnScNoPisa}) is the unique function such that whenever $f$ is given (in an appropriate function space) and $U$ is the unique solution of
	\[
	\begin{cases}
		F(U)=f\ &\text{in}\ \Om\\
		U=0\ &\text{on}\ \partial\Om,
	\end{cases}
	\]
	then $U$ is uniquely represented as
	\[
	U(x)=\int_{\Om}f(y)G(x,y)dy.
	\]
\end{DEF}

It is a standard result that if $\partial \Om$ is of class $C^k$, then the tangent bundle $T(\partial \Om)$ is of class $C^{k-1}$, consequently so is the Riemannian metric induced on $\partial \Om$ by the canonical metric on $\real^{n+1}$. It can then be seen that the Riemannian exponential mapping and the geodesic distance squared on $\partial \Om$ are respectively of class $C^{k-2}$ and $C^{k-1}$ (see \cite[Footnotes, Chapter II, Section 2]{Sakai96}).

In this paper, we will use the same characterization of a H\"older continuous vector field on $\partial \Om$ which is used in \cite{GuSc-2016MinMaxNonlocalarXiv}. We record it here for convenience. 
\begin{DEF}\label{def:HolderVectorField}
If $V: \partial \Om \to T(\partial \Om)$ is a vector field defined on $\partial \Om$, we say $V\in C_{loc}^\alpha(\partial \Om)$ if for any point $x_0\in \partial \Om$, there exists an open neighborhood $\mathcal{O}$ of $x_0$ and a constant $C>0$ such that
\begin{align*}
 \abs{V(x)-P_{y\to x}V(y)}_g\leq Cd(x, y)^\alpha,\quad \forall x, y\in \mathcal{O},
\end{align*}
where $P_{y\to x}$ is the parallel transport of a vector in $T_y(\partial \Om)$ to $T_x(\partial \Om)$, along the unique geodesic from $y$ to $x$, defined by the Levi-Civita connection of the induced Riemannian metric on $\partial \Om$.
\end{DEF}

%%%%%%%%%%%%%%%%%%%%%%%%%%%%%%%%%%%%%%%%%%%%%%
%%%%%%%%%%%%%%%%%%%%%%%%%%%%%%%%%%%%%%%%%%%%%%
%%%%%%%%%%%%%%%%%%%%%%%%%%%%%%%%%%%%%%%%%%%%%%

\subsection{Background}\label{sec:Background}

The simplest possible case of our map, $\I$,  in (\ref{eqIn:DefOfDtoN}) is when $\Om=\real^{n+1}_+$ (the upper half space), and $F(U)=\Delta U$.  This means that $U_\phi$ is the harmonic extension of $\phi$, and it is well known that $\I(\phi)=-(-\Delta)^{1/2}\phi$.  This corresponds to the generator of the boundary process, after a time rescaling, of a reflected Brownian motion in $\real^{n+1}_+$, recording the locations of the process restricted to the plane $\real^{n}\times\{0\}$.  This well known fact was generalized to bounded domains, $\Om$, as above, by Hsu in \cite{Hsu-1986ExcursionsReflectingBM}, which characterizes the generator of this boundary process as $\I$, and gives some properties, such as those in Theorem \ref{thm:MainLinear}, above.  This is in the context of the well-known relationship between D-to-N mappings and generators for boundary processes of general reflected diffusions (rescaled using their local time), and some good references are e.g. \cite[Sec. 8]{SatoUeno-1965MultiDimDiffBoundryMarkov} and \cite[Chp. IV, Sec. 7]{IkedaWatanabe-1981SDE}.  Thus, one can see Theorem \ref{thm:MainLinear} as a generalization of \cite{Hsu-1986ExcursionsReflectingBM} to more general diffusion processes with H\"older diffusion coefficients.

There is, however, a different reason for our goals in this paper beyond simply to extend \cite{Hsu-1986ExcursionsReflectingBM} to more general linear and nonlinear settings.  This is the desire to give a more precise link between D-to-N mappings and integro-differential equations, with the hopes of leveraging new results for integro-differential operators.  Developments of approximately the last 20 years have led to good understanding of the regularity for solutions of equations that involve linear and fully nonlinear integro-differential operators similar to (\ref{eqIn:IntDiffLinear})-- at least in the case that $\partial\Om=\real^n$.  Thus, it seems reasonable to further pursue the link between the integro-differential theory and D-to-N mappings, with the hope that recent results in the integro-differential theory could possibly lead to new understanding or results involving Neumann problems.  Two of the developments in the integro-differential world that could be of use are, broadly speaking: regularity results that use only the roughest bounds on coefficients and L\'evy measures-- we can call these Krylov-Safonov type estimates (we mention some specific results in the next paragraphs); and the recent result of two of the authors that shows that under certain conditions (established below) that the D-to-N mapping for fully nonlinear equations can be represented as a min-max over linear integro-differential operators \cite{GuSc-2016MinMaxNonlocalarXiv}.   In order to connect these two developments, one must, of course, gain further information about the $\mu$ (or $\mu^{ij}$) that appear in Theorems \ref{thm:MainLinear} and \ref{thm:MainNonlinear}.

  In its simplest presentation, a Krylov-Safonov result basically says that for a linear operator such as in (\ref{eqIn:IntDiffLinear}),  the solutions, say $u$, of 
\[
Lu=f\ \ \text{in}\ \ B_1
\]
satisfy the H\"older estimate, for a universal $C$,
\begin{align}\label{eqIn:KSgeneric}
[u]_{C^\al(B_{1/2})}\leq C( \norm{u}_{L^\infty(\real^n)}+\norm{f}_{L^\infty(B_1)}).
\end{align}
This has been pursued under various lists of assumptions from many various authors, and we list some explicitly below.
This estimate may seem simple, but it's importance as one of the few compactness tools for non-divergence form equations cannot be overstated.  This result was a cornerstone of the local, second order, elliptic theory, dating back the the original work of Krylov-Safonov \cite{KrSa-1980PropertyParabolicEqMeasurable}.

In recent years, Krylov-Safonov type results have been obtained for nonlocal operators like (\ref{eqIn:IntDiffLinear}) by many authors, and here we mention some of the results in this direction, and we indicate that this list is by no means complete.  Bass-Levin \cite{BaLe-2002Harnack} proved (\ref{eqIn:KSgeneric}) for the class where (for $\al\in(0,2)$)
\begin{align}\label{eqIn:BassLevinBounds}
	b\equiv0,\ \ \mu(x,dh)=k(x,h)dh,\ \ k(x,-h)=k(x,h),\ \  \text{and}\ \ 
	\frac{\lam}{\abs{h}^{n+\al}}\leq k(x,h)\leq \frac{\Lam}{\abs{h}^{n+\al}}.
\end{align}
Bass-Kassmann \cite{BaKa-05Holder}, Song-Vondracek \cite{SongVondracek-2004harnack}, and subsequently Silvestre \cite{Silv-2006Holder} (also including slightly more general $k$) extended this to the same setting, except that variable exponents, $\al(x)$, could be allowed:
\begin{align*}
	\frac{\lam}{\abs{h}^{n+\al(x)}}\leq k(x,h)\leq \frac{\Lam}{\abs{h}^{n+\al(x)}},\ \ 
	\text{for}\ \ \al(x)\in(0,2-c),\ c>0.
\end{align*} 
Finally, along this line of attack, with similar assumptions as in (\ref{eqIn:BassLevinBounds}), Caffarelli-Silvestre \cite{CaSi-09RegularityIntegroDiff} obtained (\ref{eqIn:KSgeneric}) for those kernels that satisfy
\begin{align*}
	k(x,-h)=k(x,h)\ \ \text{and}\ \ \frac{(2-\al)\lam}{\abs{h}^{n+\al}}\leq k(x,h)\leq \frac{(2-\al)\Lam}{\abs{h}^{n+\al}},
\end{align*}
and furthermore, their proof obtained the result (\ref{eqIn:KSgeneric}) in a way that is independent of $\al$ close to $2$ (the assumption that includes the factor $(2-\al)$ is consistent with the $\al/2$-Laplacian).  This made \cite{CaSi-09RegularityIntegroDiff} the first integro-differential result to contain the original result of Krylov-Safonov as a limiting case (as $\al\to2$).

The five previously mentioned works (\cite{BaKa-05Holder}, \cite{BaLe-2002Harnack}, \cite{CaSi-09RegularityIntegroDiff}, \cite{Silv-2006Holder}, \cite{SongVondracek-2004harnack}) have been generalized in approximately three overlapping directions: (i) relaxing the symmetry assumption, $k(x,-h)=k(x,h)$ in (\ref{eqIn:BassLevinBounds}); (ii) relaxing the lower bounds, $\lam\abs{h}^{-d-\al}\leq k(x,h)$, in (\ref{eqIn:BassLevinBounds}); and (iii) extending the theory to include parabolic equations.  Results that have relaxed requirements on the symmetry of $k$ include: Chang Lara \cite{Chan-2012NonlocalDriftArxiv}, Chang Lara - D\'avila \cite{ChDa-2012NonsymKernels} and \cite{ChangLaraDavila-2016HolderNonlocalParabolicDriftJDE}, Schwab-Silvestre \cite{SchwabSilvestre-2014RegularityIntDiffVeryIrregKernelsAPDE}.  Results that have relaxed requirements on the lower bounds on $k$ include: Bjorland-Caffarelli-Figalli \cite{BjCaFi-2012NonlocalGradDepOps}, Guillen-Schwab \cite{GuSc-12ABParma}, Kassmann-Mimica \cite{KassmannMimica-2013NondegJumpStochProApp}, Kassmann-Rang-Schwab \cite{KassRangSchwa-2014RegularityDirectionalINDIANA}, and \cite{SchwabSilvestre-2014RegularityIntDiffVeryIrregKernelsAPDE}.  Results that have extended the above to the parabolic setting include: \cite{ChDa-2012RegNonlocalParabolicCalcVar}, \cite{ChangLaraDavila-2016HolderNonlocalParabolicDriftJDE}, and \cite{SchwabSilvestre-2014RegularityIntDiffVeryIrregKernelsAPDE}.  Finally, we note that there is an extension that is completely separate from all of the others listed here in that it obtains Krylov-Safonov estimates in the situation that the exponent, $\al$, in (\ref{eqIn:BassLevinBounds}) is allowed to go down to $\al=0$ as well as allows for scaling laws that are more general than (\ref{eqIn:BassLevinBounds}); this is the work of Kassmann-Mimica \cite{KassmannMimica-2013IntrinsicScalingJEMS}, followed up by the work of Kim-Kim-Lee \cite{KimKimLee-2014RegularityRegularlyVaryingPOTENANALYS} .

There are many uses for the D-to-N, and we would like to point out the work of Hu-Nicholls \cite{HuNicholls-2005AnalyticityDtoNSIMA}, where they study the dependence of the D-to-N on changes to the domain, $\Om$ (for a slightly different family of equations).  There are also many useful references for related issues in \cite{HuNicholls-2005AnalyticityDtoNSIMA}.

We conclude this section by mentioning that only in the simplest setting that $\Om=\real^{n+1}_+$ and $F$ is given by (\ref{eqIn:BulkFDiv}) or (\ref{eqIn:BulkFNondiv}) will some of the above mentioned results involving non-symmentric $k$ apply to the operator $\I$ that results from Theorem \ref{thm:MainLinear}.  In the case that $F$ is nonlinear or in all cases when $\partial\Om$ is not flat, none of the above mentioned results apply to $\I$.  This suggests room for more study on this issue, and we briefly elaborate on this in Section \ref{sec:Open}.

%%%%%%%%%%%%%%%%%%%%%%%%%%%%%%%%%%%%%%%%%%%%%%
%%%%%%%%%%%%%%%%%%%%%%%%%%%%%%%%%%%%%%%%%%%%%%
%%%%%%%%%%%%%%%%%%%%%%%%%%%%%%%%%%%%%%%%%%%%%%
%%%%%%%%%%%%%%%%%%%%%%%%%%%%%%%%%%%%%%%%%%%%%%
%%%%%%%%%%%%%%%%%%%%%%%%%%%%%%%%%%%%%%%%%%%%%%
\section{Some Useful Tools For Boundary Behavior}\label{sec:Tools}

In this section, we collect some various results that will be useful later on.
The following proposition is about the boundary behavior of the Green's function for $C^{1,\al}$ domains.  The upper bound is a special case of the estimates for equations with H\"older coefficients in nice domains that can be found in Gr\"uter-Widman \cite{GruterWidman-1982GreenFunUnifEllipticManMath}.  The lower bound is a consequence of the Harnack inequality and is outlined in the proof of the main result of Zhao \cite{Zhao-1984UniformBoundednessConditionalGaugeCMP}.

\begin{prop}[Constant Coefficient]\label{propBG:GreenFunctionConstantCoef}
	Assume that $\partial\Om$ is a $C^{1,\al}$ boundary. For the constant coefficient operator, i.e. $Lu=\Delta U$, it holds that for the Green's function, $G(x,y)$, for all $x,y\in\Om$
	\begin{align*} 
		c_1\frac{d(x)d(y)}{\abs{x-y}^{n+1}}
		\leq G(x,y)
		\leq c_2\frac{d(x)d(y)}{\abs{x-y}^{n+1}}.
	\end{align*} 
	Here we use $d(x)=d(x,\partial\Om)$.  ($d(x,\partial\Om)=\inf_{y\in\partial\Om}\abs{x-y}$, and recall, $\Om\subset\real^{n+1}$)
\end{prop}

After taking normal derivatives of the Green's function, this gives in \cite{Zhao-1984UniformBoundednessConditionalGaugeCMP},

\begin{prop}[Poisson Kernel Constant Coefficients, \cite{Zhao-1984UniformBoundednessConditionalGaugeCMP}]\label{propBG:PoissonKernelConstCoeff}
	Assume that $\partial\Om$ is a $C^{1,\al}$ boundary. For the constant coefficient operator, i.e. $LU=\Delta U$, it holds that for the Poisson kernel, $P(x,y)$, for all $x\in\Om$ and $z\in\partial\Om$
	\begin{align*} 
		c_1\frac{d(x)}{\abs{x-z}^{n+1}}
		\leq P(x,z)
		\leq c_2\frac{d(x)}{\abs{x-z}^{n+1}}.
	\end{align*}
	(Recall, $\Om\subset\real^{n+1}$)
\end{prop}

It turns out that the same behavior was extended to variable coefficients by respectively Cho \cite{Cho-2006TwoSidedGlobalEstGreenFunDParabolicDiv} and Hueber-Sieveking \cite{HueberSieveking-1982UniformBoundsQuotientsGreenFunctions}.  We record this here

\begin{prop}[Variable Coefficients]\label{propBG:GreenFunctionVariableCoef}
	\begin{enumerate}[(a)]
		\item (Hueber-Sieveking \cite{HueberSieveking-1982UniformBoundsQuotientsGreenFunctions}) Assume that
		\[
		LU=\Tr(A(x)D^2U(x)) + B(x)\cdot\grad U(x) + C(x)U(x),
		\]
		with H\"older coefficients and that $\partial\Om$ is $C^{1,1}$.  Then Proposition \ref{propBG:GreenFunctionConstantCoef} remains true.
		\item (Cho \cite{Cho-2006TwoSidedGlobalEstGreenFunDParabolicDiv}) Assume that 
		\[
		LU=\Div(A(x)\grad U(x)),
		\]
		with H\"older coefficients, and that $\partial\Om$ is $C^{1,\al}$. Then Proposition \ref{propBG:GreenFunctionConstantCoef} remains true.
	\end{enumerate}
\end{prop}

We note that Cho \cite{Cho-2006TwoSidedGlobalEstGreenFunDParabolicDiv} proves the estimate for the Heat kernel, but the result for the elliptic problem follows from the identity
\begin{align*}
	G(x,y) = \int_0^\infty p(x,y,t)dt,
\end{align*}
where $p(x,y,t)$ is the heat kernel, or transition density function for the corresponding killed process in $\Om$ (i.e. the fundamental solution of the heat equation with zero boundary data).

Finally, we state here a relationship between the $F$-Harmonic measure and the Green's function for a linear equation.  We note that the Green's function for non-divergence equations are well known not to be well behaved pointwise; however, in light of the fact that we are dealing with equations with H\"older coefficients, this is a situation where the Green's function is defined pointwise, furthermore we record the actual result we use in the next lemma.

\begin{prop}[Harmonic Measure - Green Function estimates]\label{propBG:HarmonicMeasureGreenFunction} Let $\{\om_x\}_{x\in\Om}$ be the $F$-harmonic measure where $F$ is defined by \eqref{eqIn:BulkFDiv} or \eqref{eqIn:BulkFNondiv}, and $G$ be the Green's function for $F$ on $\Om$.

Then there are universal constants $\rho_0$, $C_1$, $C_2>0$ and $s_0>1$ such that for any $\rho\in (0, \rho_0)$, $x\in\partial \Om$, and $y\in \Om\setminus B_{s_0 \rho}(x)$, for the divergence equation (\ref{eqIn:BulkFDiv}) it holds
\begin{align*}
C_1\rho^{n-1}G(y, x+\rho \nu(x))\leq \omega_{y}(\partial\Om\cap B_\rho(x))\leq C_2\rho^{n-1}G( y,x+\rho \nu(x)),
\end{align*}
and for the non-divergence equation (\ref{eqIn:BulkFNondiv}) it holds that
\begin{align*}
	\frac{C_1}{\rho^2}\int_{\tilde B_\rho}G(y,z)dz\leq 
	\omega_{y}(\partial\Om\cap B_\rho(x))\leq
	\frac{C_2}{\rho^2}\int_{B^{n+1}_\rho(x)\intersect\Om} G(y,x)dz
\end{align*}
where $\tilde B_\rho=B^{n+1}(x+\frac{1}{4}\nu(x))$
\end{prop}

\begin{proof}
 The divergence case \eqref{eqIn:BulkFDiv} is an immediate consequence of \cite[Lemma 2.2]{CaffFabesMortSalsa-1981BoundBehavEllipticIUMJ}.

 For the non-divergence case, both bounds appear in the proof of \cite[Lemma 5.18]{Kenig-1993PotentialThoeryNonDiv}.  We also use the lower bound explicitly as a crucial step in Section \ref{sec:FullyNonlinear}, and so the proof of that inequality appears in the proof of Lemma \ref{lemFN:HarmonicMeasureGreenIntegralEstimate}, below.
 
\end{proof}

\begin{lem}[Comparison of intrinsic and extrinsic annuli]\label{lemTOOL:RingsAmbientIntrinsic}
There exists an $\epsilon_0>0$ such that for any $r\in (0, \epsilon_0)$ and $x_0\in\partial\Om$,
	\begin{align*}
		(B^{n+1}_{(7/4)r}(x_0)\setminus B^{n+1}_{(5/4)r}(x_0))\intersect\partial\Om
		\subset
		(B_{2r}(x_0)\setminus B_r(x_0))
		\subset
		(B^{n+1}_{(9/4)r}(x_0)\setminus B^{n+1}_{(3/4)r}(x_0))\intersect\partial\Om.
	\end{align*}
\end{lem}
\begin{proof}
By definition of geodesic distance, it is clear that for any $h$, $x_0\in \partial \Om$,
\begin{align*}
 \abs{x_0- h}\leq d(x_0, h).
\end{align*}
Now since $\partial \Om$ is $C^2$, for any $x\in \partial \Om$ there exists $\epsilon_x>0$ and $\rho_x\in C^2(B^n_{\epsilon_x}(x))$ with (after a rotation of coordinates)
\begin{align*}
 B^{n+1}_{\epsilon_x}(x)\cap \Om&=\{(y', y^{n+1})\in B^{n+1}_{\epsilon_x}(x)\ :\  \rho_x(y')<y^{n+1}\},\\
 \rho_x(x)=0,&\quad \grad \rho_x(x)=0,\\
 \sqrt{1+\norm{\rho_x}_{C^1(B^n_{\epsilon_x}(x))}^2}&\leq \frac{9}{8}.
\end{align*}
By compactness of $\partial \Om$ we can see $\epsilon_1:=\inf_{x\in\partial\Om}\epsilon_x>0$. Now let $x_0\in \partial \Om$, rotate coordinates to identify $\real^n$ with $T_{x_0}(\partial \Om)$. We claim that the projection $\pi_{\real^n}(B_{\epsilon_1}(x_0))$ into $\real^n$ is contained in $B^n_{\epsilon_1}(x_0)$. Suppose this is not the case, so $h=(h', h^{n+1})\in B_{\epsilon_1}(x_0)$ but $h'\not\in B^n_{\epsilon_1}(x_0)$. We can take a length minimizing $C^1$ curve $\gamma: [0, 1]\to \partial \Om$ connecting $x_0$ to $h$, which we assume constant speed, and let $t_0:=\inf\{t\in [0, 1]\mid \pi_{\real^n}(\gamma(t))\not \in B^n_{\epsilon_1}(x_0)\}>0$. Then we calculate
\begin{align*}
\epsilon_1> d(x_0, h)\geq \int_0^{t_0}\abs{\dot\gamma(t)}dt\geq \int_0^{t_0}\abs{\pi_{\real^n}(\dot\gamma(t))}dt\geq \epsilon_1,
\end{align*}
a contradiction.

Thus if $h\in B_{\epsilon_1}(x_0)$, we can write $h:=(y', \rho_{x_0}(y'))$ for some $y'\in B^n_{\epsilon_1}(x_0)$, then
\begin{align*}
 d(x_0, h)\leq \int_0^1\abs{(y', (\grad\rho_{x_0}(ty')\cdot y'))}dt\leq \sqrt{\abs{y'}^2+\norm{\rho_{x_0}}_{C^1(B^n_{\epsilon_1}(x_0))}^2\abs{y'}^2}\leq \frac{9}{8}\abs{y'}\leq \frac{9}{8}\abs{x_0-h},
\end{align*}
and the claimed inclusions immediately follow.
\end{proof}

%%%%%%%%%%%%%%%%%%%%%%%%%%%%%%%%%%%%%%%%%%%%%%
%%%%%%%%%%%%%%%%%%%%%%%%%%%%%%%%%%%%%%%%%%%%%%
%%%%%%%%%%%%%%%%%%%%%%%%%%%%%%%%%%%%%%%%%%%%%%
%%%%%%%%%%%%%%%%%%%%%%%%%%%%%%%%%%%%%%%%%%%%%%
%%%%%%%%%%%%%%%%%%%%%%%%%%%%%%%%%%%%%%%%%%%%%%

%%%%%%%%%%%%%%%%%%%%%%%%%%%%%%%%%%%%%%%%%%%%%%
%%%%%%%%%%%%%%%%%%%%%%%%%%%%%%%%%%%%%%%%%%%%%%
%%%%%%%%%%%%%%%%%%%%%%%%%%%%%%%%%%%%%%%%%%%%%%
%%%%%%%%%%%%%%%%%%%%%%%%%%%%%%%%%%%%%%%%%%%%%%
%%%%%%%%%%%%%%%%%%%%%%%%%%%%%%%%%%%%%%%%%%%%%%

\section{Well-Posedness and Lipschitz Nature of the D-to-N}

Here we record the relatively straightforward facts that $\I$ defined via (\ref{eqIn:BulkExtensionGeneric}) and (\ref{eqIn:DefOfDtoN}) is in fact well defined and a Lipschitz mapping $C^{1,\al}\to C^\al$ in each of the three instances (\ref{eqIn:BulkFDiv}), (\ref{eqIn:BulkFNondiv}), (\ref{eqIn:BulkFNonlinear}).

\begin{lem}\label{lemWP:DtoNWellDefined}
	If  the equation (\ref{eqIn:BulkExtensionGeneric}) satisfies the assumptions 
	\begin{enumerate}[(i)]
		\item $\displaystyle \phi\in C^{1,\al}(\partial\Om)\ \implies\ U_\phi\in C^{1,\al'}(\overline\Om)\ \text{for some}\ 0<\al'<\al \ \text{(Regularity)}$
		\item $\displaystyle \phi\leq\psi\ \text{on}\ \partial\Om\ \implies\ U_\phi\leq U_\psi\ \text{in}\ \Om\ \ \text{(Comparison)}$
		\item $\displaystyle \phi\in C(\partial\Om)\ \implies U_\phi\ \text{exists and is unique in the (weak, viscosity, strong, or classical sense)}$,
	\end{enumerate}
then the D-to-N mapping, $\I$, defined in (\ref{eqIn:BulkExtensionGeneric}) and (\ref{eqIn:DefOfDtoN}) is well defined and has the global comparison property over $C^{1,\al}(\partial\Om)$.
\end{lem} 

\begin{proof}
	First of all, the assumption of existence and uniqueness of $U_\phi$, combined with the assumption (Regularity) at least show that $\I$ is well defined as a map from $C^{1,\al}(\partial\Om)$ to $C^{\al'}(\partial\Om)$.  The only thing to check is the comparison property.  However, $\I$ inherits this directly from the assumption (Comparison) that is made on $F$.   Indeed, let $u$, $v$, and $x\in\partial\Om$ be given such that $u\leq v$ on $\partial\Om$ and that $u(x)=v(x)$.  Let $\nu(x)$ be the \emph{inward} normal vector at $x$ and let $h>0$ be small enough.  Thus by (Comparison), we see that
	\begin{align*}
		U_u(x+h\nu(x))-U_u(x)\leq U_v(x+h\nu(x))-U_v(x),
	\end{align*}
	and thus since $\partial_\nu U_u$ and $\partial_\nu U_v$ exist by (Regularity), we conclude 
	\begin{align*}
		\partial_\nu U_u(x)\leq\partial_\nu U_v(x).
	\end{align*}
\end{proof}

Just for completeness, we include a list of results which establish the assumptions of comparison and regularity made in Lemma \ref{lemWP:DtoNWellDefined} for each of the cases of $F$ in (\ref{eqIn:BulkFDiv})--(\ref{eqIn:BulkFNonlinear}).

\begin{lem}\label{lemWP:FSatisfiesComparisonAndRegularity}
	If $F$ is given as (\ref{eqIn:BulkFDiv}), (\ref{eqIn:BulkFNondiv}), or (\ref{eqIn:BulkFNonlinear}), then the equation (\ref{eqIn:BulkExtensionGeneric}) satisfies the assumptions of (regularity), (comparison), (existence/uniqueness) listed in Lemma \ref{lemWP:DtoNWellDefined}.
\end{lem}

\begin{proof}[Proof of Lemma \ref{lemWP:FSatisfiesComparisonAndRegularity}]
	In the case of (\ref{eqIn:BulkFDiv}), weak solutions are defined via the bilinear form,
	\begin{align*}
		B(u,v)= \int_{\Om}\grad u(x)\cdot A(x)\grad v(x)dx,
	\end{align*}
	and the establishment of uniqueness, comparison, and regularity under the assumption that $A\in C^\al(\Om)$ can be found in \cite[Chp 8]{GiTr-98}.
	
	In the case of (\ref{eqIn:BulkFNondiv}), ``weak'' solutions can be understood as either strong solutions e.g. \cite[Chp 9]{GiTr-98} or viscosity solutions e.g. \cite{CrIsLi-92} (both cases are equivalent for this equation and these assumptions).  Note, in this case, $U_\phi$ is actually $C^{2,\al}_{loc}(\Om)$, but not in the whole of $\overline{\Om}$ as we only assume $\phi\in C^{1,\al}(\partial\Om)$.  The assumptions that $A\in C^\al(\Om)$ and is uniformly elliptic imply uniqueness, comparison, and regularity, and can be found in \cite[Chp 9]{GiTr-98}, among other sources.
	
	Finally, in the case of (\ref{eqIn:BulkFNonlinear}), the ``locally H\"older coefficients'' assumption means that for all symmetric matrices, $P$,
	\begin{align*}
		\abs{F(P,x)-F(P,y)}\leq C\abs{x-y}^\al(1+\norm{P}),
	\end{align*}
	and for simplicity we can assume that $F(0,x)\equiv0$. The notion of weak solution is viscosity solutions, e.g. \cite{CrIsLi-92}. We refer to \cite[Theorem 1.4]{SilvestreSirakov-2013boundary} for the validity of the $C^{1,\al}$ estimates in (regularity), and to \cite[Theorem III.1]{IshiiLions-1990ViscositySolutions2ndOrder} for the validity of the comparison result, which in this context also gives the uniqueness of the viscosity solution.
\end{proof}

Just as in the case of second order elliptic equations, it will be useful to understand which operators govern the ellipticity class for the D-to-N, $\I$, in the context of $F$ in (\ref{eqIn:BulkFNonlinear}) (i.e. the analogous objects to the Pucci operators for second order equations that appear in Definition \ref{def:MaximalPucci}).  It turns out that a convenient choice of these extremal operators are the D-to-N operators for the second order extremal operators.  The following observation is copied from \cite[Lemma 3.3]{GuSc-2014NeumannHomogPart1DCDS-A}:

\begin{lem}\label{lemWP:DtoNExtremals}
	In (\ref{eqIn:BulkExtensionGeneric}), take $F$ to be respectively $\M^-$ and $\M^+$ which are in Definition (\ref{def:MaximalPucci}), and take respectively $U^-_\phi$ and $U^+_\phi$ to be the corresponding solutions of (\ref{eqIn:BulkExtensionGeneric}).  Define the boundary extremal operators as
	\begin{align}\label{eqWP:ExtremalForBoundaryOpDef}
		M^-(\phi,x):= \partial_\nu U^-_\phi(x)\ \text{and}\ M^+(\phi,x):=\partial_\nu U^+_\phi(x).
	\end{align}
	Then $M^\pm$ are extremal operators for $\I$ in the sense that for all $u,v\in C^{1,\al}(\partial\Om)$ and for all $x\in\partial\Om$
	\begin{align}\label{eqWP:ExtremalInequalityBoundaryOps}
		M^-(u-v,x)\leq \I(u,x)-\I(v,x)\leq M^+(u-v,x).
	\end{align}
\end{lem}

\begin{proof}
Here $F$ is a fixed uniformly elliptic operator from (\ref{eqIn:BulkFNonlinear}).  For ease of presentation, we record the two different equations that are being used here:
	\begin{align}\label{eqWP:ExtremalLemWithRegularF}
		\begin{cases}
			F(D^2U,x)=0\ &\text{in}\ \Om\\
			U=\phi\ &\text{on}\ \partial\Om.
		\end{cases}
	\end{align}  
	and
	\begin{align}\label{eqWP:ExtremalLemWithMaximalPucciF}
		\begin{cases}
			\M^+(U,x)=0\ &\text{in}\ \Om\\
			U=\phi\ &\text{on}\ \partial\Om.
		\end{cases}
	\end{align}
		Let $U_u$ and $U_v$ be the unique solutions of (\ref{eqWP:ExtremalLemWithRegularF}) with respectively boundary data given by $\phi=u$ and $\phi=v$.  We will just prove the upper bound, and the lower bound follows analogously.

	We note that since $U_u$ and $U_v$ are respectively a viscosity sub and super solution of (\ref{eqWP:ExtremalLemWithMaximalPucciF}), then it follows that $U_u-U_v$ is a viscosity subsolution of 
	\begin{align*}
		0\leq \M^+(U_u-U_v).
	\end{align*}
	Hence, if $U^+_{(u-v)}$ is the solution to (\ref{eqWP:ExtremalLemWithMaximalPucciF}) with $\phi=u-v$, since $U^+_{(u-v)}$ and $U_u-U_v$ have the same boundary data, the comparison of sub and super solutions for (\ref{eqWP:ExtremalLemWithMaximalPucciF}) shows that
	\begin{align*}
		U_u-U_v\leq U^+_{(u-v)}\ \text{in}\ \Om\ \text{and}\ U_u-U_v=u-v=U^+_{(u-v)}\ \text{on}\ \partial\Om.
	\end{align*}
	Hence
	\begin{align*}
		\partial_\nu U_u-\partial_\nu U_v\leq \partial_\nu U^+_{(u-v)}=M^+(u-v),
	\end{align*}
	which concludes the lemma.
\end{proof}

\begin{lem}\label{lemWP:DtoNLipschitz}
	In all cases of (\ref{eqIn:BulkFDiv}), (\ref{eqIn:BulkFNondiv}), (\ref{eqIn:BulkFNonlinear}), there exists some choice of $\al'$ with $0<\al'<\al$, so the D-to-N, $\I$, is a Lipschitz mapping of $C^{1,\al}(\partial\Om)\to C^{\al'}(\partial\Om)$. $\I$ also satisfies the extra assumption  in \cite[Theorem 1.6 - (1.3)]{GuSc-2016MinMaxNonlocalarXiv}, which requires
	\begin{align}\label{eqWP:ExtraAssumptionNonlinearCourrege}
		\forall\ u,v\in C^{1,\al}(\partial\Om),\ \ 
		\norm{\I(u)-\I(v)}_{L^\infty(B_{r})}
		\leq C\left(\norm{u-v}_{C^{1,\al}(\overline{B_{2r}})} + \om(r)\norm{u-v}_{L^\infty(\partial\Om)}\right),
	\end{align}
	and $\om(r)\to0$ as $r\to\infty$.
\end{lem}

\begin{proof}
	First, we remark on the special assumption (1.3) in \cite[Theorem 1.6]{GuSc-2016MinMaxNonlocalarXiv}, which we listed here as (\ref{eqWP:ExtraAssumptionNonlinearCourrege}).  In this context, we simply require that the normal derivative of the solution, $U_u$, in $B_r$ is controlled by $\norm{u}_{C^{1,\al}(B_{2r})}$, which is a standard type of estimate for boundary regularity.  We recall that we are assuming for simplicity that $\Om$ is bounded.  Hence, (\ref{eqWP:ExtraAssumptionNonlinearCourrege}) is trivial once the Lipschitz character of $\I$ is established, as we can just take $\om(r)\equiv0$ once $r>\diam(\Om)$.

	The Lipschitz nature of $\I$ follows from the global (up to the boundary) $C^{1,\al'}$ regularity theory for (\ref{eqIn:BulkExtensionGeneric}). Let $u,v\in C^{1,\al}(\partial\Om)$.
	In the two linear cases, (\ref{eqIn:BulkFDiv}) and (\ref{eqIn:BulkFNondiv}), we note that (with apologies for the triviality)
	\begin{align*}
		\I(u)-\I(v)=\I(u-v),
	\end{align*}
	and in the nonlinear case (\ref{eqIn:BulkFNonlinear}) that we will invoke the extremal inequalities (\ref{eqWP:ExtremalInequalityBoundaryOps}), which means we will be utilizing boundary regularity theory for
	\begin{align}\label{eqWP:PucciEqsForLipLemma}
		\M^-(U_{u-v},x)=0,\ \ \text{and}\ \ \M^+(U_{u-v},x)=0\ \ \text{in}\ \Om.
	\end{align}
	
	For the divergence case, (\ref{eqIn:BulkFDiv}), one reference is \cite[Theorem 8.33]{GiTr-98}, and for the non-divergence case, (\ref{eqIn:BulkFNondiv}), the regularity is a straightforward consequence for the boundary oscillation reduction of the quantity $U(x)/d(x,\partial\Om)$ that can be found in \cite[Theorem 9.31]{GiTr-98}.  For the nonlinear case (\ref{eqIn:BulkFNonlinear}) one reference is \cite[Theorem 1.1]{SilvestreSirakov-2013boundary}, applied to each of the equations in (\ref{eqWP:PucciEqsForLipLemma}).  All of these results imply that for a universal $C$
		\begin{align}\label{eqDtoN:C1GammaEstimate}
			\norm{U_\phi}_{C^{1,\al'}(\overline{\Om})}\leq C\left( \norm{U_\phi}_{L^\infty(\overline{\Om})} + \norm{\phi}_{C^{1,\al}(\partial\Om)}    \right),
		\end{align}
		and when combined with the maximum principle, $\abs{U_\phi}\leq \norm{\phi}_{L^\infty(\partial\Om)}$, we see that
		\begin{align*}
			\norm{\I(\phi,\cdot)}_{C^{\al'}(\partial\Om)}\leq C\norm{\phi}_{C^{1,\al}(\partial\Om)}.
		\end{align*}
		Hence, applying this in each of our cases to $\phi=u-v$, we obtain the Lipschitz bound.
\end{proof}

%%%%%%%%%%%%%%%%%%%%%%%%%%%%%%%%%%%%%%%%%%%%%%
%%%%%%%%%%%%%%%%%%%%%%%%%%%%%%%%%%%%%%%%%%%%%%
%%%%%%%%%%%%%%%%%%%%%%%%%%%%%%%%%%%%%%%%%%%%%%
%%%%%%%%%%%%%%%%%%%%%%%%%%%%%%%%%%%%%%%%%%%%%%
%%%%%%%%%%%%%%%%%%%%%%%%%%%%%%%%%%%%%%%%%%%%%%

%%%%%%%%%%%%%%%%%%%%%%%%%%%%%%%%%%%%%%%%%%%%%%
%%%%%%%%%%%%%%%%%%%%%%%%%%%%%%%%%%%%%%%%%%%%%%
%%%%%%%%%%%%%%%%%%%%%%%%%%%%%%%%%%%%%%%%%%%%%%
%%%%%%%%%%%%%%%%%%%%%%%%%%%%%%%%%%%%%%%%%%%%%%
%%%%%%%%%%%%%%%%%%%%%%%%%%%%%%%%%%%%%%%%%%%%%%

\section{Linear equations with H\"older coefficients-- Proofs of Theorems \ref{thm:MainLinear}, \ref{thm:HolderDrift}, \ref{thm:HolderXDependenceTVNorm}}\label{sec:LinearEq}

%%%%%%%%%%%%%%%%%%%%%%%%%%%%%%%%%%%%%%%%%%%%%%
%%%%%%%%%%%%%%%%%%%%%%%%%%%%%%%%%%%%%%%%%%%%%%
%%%%%%%%%%%%%%%%%%%%%%%%%%%%%%%%%%%%%%%%%%%%%%
%%%%%%%%%%%%%%%%%%%%%%%%%%%%%%%%%%%%%%%%%%%%%%
%%%%%%%%%%%%%%%%%%%%%%%%%%%%%%%%%%%%%%%%%%%%%%

In this section, we include the proofs of Theorems \ref{thm:MainLinear}, \ref{thm:HolderDrift}, \ref{thm:HolderXDependenceTVNorm}.  We note that the existence of $b$ and $\mu$, the validity of (\ref{eqIn:Thm1DtoNIntDiffLinear}), boundedness of $b$ are all a direct result of \cite[Theorem 1.6 and Proposition 1.7]{GuSc-2016MinMaxNonlocalarXiv}.

%%%%%%%%%%%%%%%%%%%%%%%%%%%%%%%%%%%%%%%%%%%%%%
%%%%%%%%%%%%%%%%%%%%%%%%%%%%%%%%%%%%%%%%%%%%%%
%%%%%%%%%%%%%%%%%%%%%%%%%%%%%%%%%%%%%%%%%%%%%%
%%%%%%%%%%%%%%%%%%%%%%%%%%%%%%%%%%%%%%%%%%%%%%
%%%%%%%%%%%%%%%%%%%%%%%%%%%%%%%%%%%%%%%%%%%%%%

\subsection{Density and bounds for $\mu$ (Proof of Theorem \ref{thm:MainLinear})}

\begin{proof}[Proof of Theorem \ref{thm:MainLinear}]
	
 Fix $x\in\partial\Om$, we show that $\mu(x,\cdot)$ is absolutely continuous with respect to surface measure, $\sig$, on $\partial \Om$ on $\partial\Om\setminus \{x\}$. This will be done by showing absolute continuity on the set $\partial \Om\setminus \{\overline{B_r(x)}\}$ for any arbitrary $r>0$, then we can exhaust $\partial\Om\setminus \{x\}$ by a union of such sets. Thus fix $r>0$ and any set $E\subset \partial \Om\setminus \{\overline{B_r(x)}\}$ with $\sig(E)=0$. 
 
 Fix $\delta>0$, then we find a countable cover $\{B(x_j, r_j)\}_{j=1}^\infty$ of $E$ by open geodesic balls such that $\sum_{j=1}^\infty r_j^n<\delta$; let us write $B_j:=B(x_j, r_j)$ for brevity. Now let $\phi\in C^2(\partial \Om)$ be any function such that 
 \begin{align*}
0\leq \phi \leq \Indicator_{\bigcup_{j=1}^\infty B_j}.
\end{align*}
 If $\delta$ is sufficiently small compared to $r$, we will have $\phi\equiv 0$ on $\overline{B_{r/2}(x)}$ thus $\grad \phi(x)=0$, so in (\ref{eqIn:Thm1DtoNIntDiffLinear}) we have 
\begin{align*}
 \I(\phi,x) = \int_{\partial\Om\setminus\overline{B_{r/2}(x)}} \phi(y)\mu(x, dy).
\end{align*}

Let $\{\om_x\}_{x\in \Om}$ be the $F$-harmonic measure for $F$ given by \eqref{eqIn:BulkFDiv} (see Definition \ref{def:HarmonicMeasure}), then recall
\begin{align*}
 U_\phi(x)=\int_{\partial \Om}\phi(y)\omega_x(dy)
\end{align*}
for any $x\in\Om$. Now if $s>0$ is sufficiently small, for each $j$ by Proposition \ref{propBG:HarmonicMeasureGreenFunction} we have

\begin{align*}
 \omega_{x+s \nu(x)}(\partial\Om\cap B_j)\leq 
 \begin{cases}
 Cr_j^{n-1}G(x+s \nu(x),x_j+r_j \nu(x_j) )\ &\text{for}\ (\ref{eqIn:BulkFDiv})\\
 \frac{C_2}{r_j^2}\int_{B^{n+1}_{r_j}(x)\intersect\Om} G(x+s \nu(x),z)dz\ &\text{for}\ (\ref{eqIn:BulkFNondiv}).
 \end{cases}
\end{align*}
where $G$ is the Green's function and $C$ depends only on $\partial \Om$ and the ellipticity of the equation. Thus we have the estimate
\begin{align}
	U_{\phi}(x+s \nu(x)) &= \int_{\partial\Om} \phi(y)\omega_{x+s \nu(x)}(dy)\nonumber\\
	&\leq \sum_{j=1}^\infty \omega_{x+s \nu(x)}(\partial\Om\cap B_j)\nonumber\\
	&\leq C\sum_{j=1}^\infty 
    \begin{cases}
    r_j^{n-1}G(x+s \nu(x),x_j+r_j \nu(x_j) )\ &\text{for}\ (\ref{eqIn:BulkFDiv})\\
    \frac{1}{r_j^2}\int_{B^{n+1}_{r_j}(x)\intersect\Om} G(x+s \nu(x),z)dz\ &\text{for}\ (\ref{eqIn:BulkFNondiv}).
    \end{cases}\nonumber\\
	&\leq C\sum_{j=1}^\infty 
	\begin{cases}
		r_j^{n-1}\frac{sr_j}{\lvert x+s \nu(x)-(x_j+r_j \nu(x_j))\rvert^{n+1}} &\text{for}\ (\ref{eqIn:BulkFDiv})\vspace{0.1in}\\
		\frac{1}{r_j^2}\cdot\frac{sr_j^{n+2}}{\lvert x+s \nu(x)-(x_j+r_j \nu(x_j))\rvert^{n+1}} &\text{for}\ (\ref{eqIn:BulkFNondiv})
	\end{cases}\label{eqLin:DensityGreenEstimate}\\
	&\leq C_r s\sum_{j=1}^\infty r_j^n<C_rs\delta\nonumber
	\end{align}
where we have used Proposition \ref{propBG:GreenFunctionVariableCoef} to obtain the second to final inequality and $C_r$ is some constant depending on $n$, $r$, ellipticity, and $\partial \Om$ (but independent of $\delta$ and $\phi$).
Thus 
\begin{align*}
\int_{\partial\Om\setminus\overline{B_{r/2}(x)}} \phi(y)\mu(x, dy)&=\partial_\nu U_{\phi}(x)=\lim_{s\to 0}\frac{U_{\phi}(x+s \nu(x))-\phi(x)}{s}\\
 &\leq C_r\delta.
\end{align*}
Since $\{B_j\}$ covers $E$, we can take a sequence of $C^2(\partial \Om)$ functions $\Indicator_{E}\leq \phi_k\leq \Indicator_{\bigcup_{j=1}^\infty B_j}$ decreasing pointwise to $\Indicator_E$ to obtain $\mu(x, E)\leq C_r\delta$, and since $\delta$ was arbitrary this yields $\mu(x, E)=0$.

By the above, we can write $\mu(x, dy)=K(x, y)\sigma(dy)$ for some density $K$ when restricted to $\partial\Om\setminus\{x\}$. Fix $y\neq x$ in $\partial \Om$ and $0<2r<\lvert x-y\rvert$, and this time let $\phi^r_l$ and $\phi^r_u\in C^2(\partial\Om)$ be such that
	\begin{align*}
		0\leq \phi^r_l\leq \Indicator_{B_{r}(y)}\leq \phi^r_u
	\end{align*}
	with
	\begin{align*}
		\begin{cases}
		\phi^r_l&\equiv 1\ \text{in}\ B_{r/2}(y)\\
		\phi^r_l&\equiv 0\ \text{in}\ \partial\Om\setminus B_{2r}(x)\\
		\phi^r_u&\equiv 1\ \text{in}\ B_{3r/2}(x)\\
		\phi^r_u&\equiv 0\ \text{in}\ \partial\Om\setminus B_{2r}(x).
		\end{cases}
	\end{align*}
Following similar calculations as before, and invoking the same split argument for the divergence/non-divergence setting in (\ref{eqLin:DensityGreenEstimate}),  we have
\begin{align*}
	U_{\phi^r_l}(x+s \nu(x)) &= \int_{\partial\Om} \phi^r_l(z)\omega_{x+s \nu(x)}(dz)\\
	&\leq \omega_{x+s \nu(x)}(\partial\Om\cap B_r(y))\\
	&\leq \frac{C_2sr^n}{\lvert x+s \nu(x)-(y+r \nu(y))\rvert^{n+1}}.
	\end{align*}
Since again $\phi^r_l\equiv 0$ near $x$, we have $\I(\phi^r_l,x) = \int_{\partial\Om\cap B_{2r}(x)} \phi^r_l(y)\mu(x, dy)$, hence taking the limit in the difference quotient we have
\begin{align*}
 \frac{\mu(x, \partial\Om\cap B_{2r}(x))}{\sigma(\partial\Om\cap B_{2r}(x))}&\leq\frac{1}{\sigma(\partial\Om\cap B_{2r}(x))}\int_{\partial\Om\cap B_{2r}(x)} \phi^r_l(y)\mu(x, dy)\\
 &=\frac{\partial_\nu U_{\phi^r_l}(x)}{\sigma(\partial\Om\cap B_{2r}(x))}\\
 &\leq \frac{C_2}{\lvert x-(y+r \nu(y))\rvert^{n+1}}.
\end{align*}
a similar calculation utilizing $\phi^r_u$ yields
\begin{align*}
 \frac{\mu(x, \partial\Om\cap B_{2r}(x))}{\sigma(\partial\Om\cap B_{2r}(x))}\geq \frac{C_1}{\lvert x-(y+r \nu(y))\rvert^{n+1}}.
\end{align*}
By the Lebesgue differentiation theorem, for $\sigma$-a.e. $y\in\partial\Om$ we have
\begin{align*}
\frac{C_1}{\lvert x-y\rvert^{n+1}}\leq K(x, y)=\lim_{r\to 0}  \frac{\mu(x, \partial\Om\cap B_{2r}(x))}{\sigma(\partial\Om\cap B_{2r}(x))}\leq \frac{C_2}{\lvert x-y\rvert^{n+1}}.
\end{align*}

\end{proof}

%%%%%%%%%%%%%%%%%%%%%%%%%%%%%%%%%%%%%%%%%%%%%%
%%%%%%%%%%%%%%%%%%%%%%%%%%%%%%%%%%%%%%%%%%%%%%
%%%%%%%%%%%%%%%%%%%%%%%%%%%%%%%%%%%%%%%%%%%%%%
%%%%%%%%%%%%%%%%%%%%%%%%%%%%%%%%%%%%%%%%%%%%%%
%%%%%%%%%%%%%%%%%%%%%%%%%%%%%%%%%%%%%%%%%%%%%%

\subsection{H\"older continuity of the coefficients of $\I$ (Proof of Theorem \ref{thm:HolderDrift})}

Before embarking on the proof of Theorem \ref{thm:MainLinear}, we make some background observations. Recall $2r_0>0$ will always be a constant smaller than the injectivity radius of $\partial \Om$. For this portion we assume $\partial \Om$ to be a $C^5$ surface, this means the tangent bundle $T(\partial \Om)$ is a $C^{4}$ manifold. Then the restriction of the Euclidean metric from $\real^{n+1}$ to $\partial \Om$ is also $C^{4}$, and the exponential mapping $\exp_x$ based at any point $x\in \partial \Om$ is $C^{3}$ (the same holds for its inverse in its domain of definition). In particular the geodesic distance squared will be $C^{4}$ on $B_{r_0}(x_0)\times B_{r_0}(x_0)$, meaning that the second derivative involving the mapping $D(\exp^{-1}_p)\vert_h$ leading to the estimate \eqref{eqn:holdergradientbound} below is justified. Finally, recall Definition \ref{def:HolderVectorField} for the H\"older continuity of a vector field on $\partial \Om$.

\begin{proof}[Proof of Theorem \ref{thm:HolderDrift}]

Fix $x_0$, $y_0\in \partial \Om$ which will be taken so $d(x_0, y_0)$ is smaller than some universal constant, that is yet to be determined. For ease of notation let us write
\begin{align*}
 d_0:=d(x_0, y_0),
\end{align*}
and we tacitly assume $d_0\leq \min\{1,r_0\}$. Also fix a unit length $v\in T_{x_0}(\partial \Om)$, and let $\phi$ be a $C^{2}$ function on $\partial \Om$ such that for $h\in B_{2r_0}(x_0)$ we have
\begin{align*}
 \phi(h)=(v, \exp^{-1}_{x_0}(h))_g.
\end{align*}

Computing using normal coordinates centered at $x_0$ we easily see $\grad\phi(x_0)=v$, and in particular $\phi(h)=(\grad \phi(x_0), \exp^{-1}_{x_0}(h))_g$ on $B_{r_0}(x_0)$. 
Also let $\eta\in C^\infty(\real)$ be such that $0\leq \eta\leq 1$, $\eta\equiv 1$ on $[0, r_0]$, and $\eta\equiv 0$ on $[r_0+d_0^{\alpha_1}, \infty)$ for some $\alpha_1\in (0, 1)$ which will be determined later; we will also assume that $r_0+d_0^{\alpha_1}$ is less than the injectivity radius of $\partial \Om$, and so $d_0^{\al_1}\leq r_0$ will suffice. 
We also define $\eta_{x_0}$, $\eta_{y_0}$ by $\eta(d(x_0, \cdot))$ and $\eta(d(y_0, \cdot))$ respectively, both of which can be seen to be $C^{3}$.
Then

\begin{align*}
&\I(\eta_{x_0}\phi,x_0) 
\\&= (b(x_0),\grad (\eta_{x_0}\phi)(x_0))_g+\int_{\partial\Om}\left( (\eta_{x_0}\phi)(h)-(\eta_{x_0}\phi)(x_0)-\Indicator_{B_{r_0}(x_0)}(h)(\grad (\eta_{x_0}\phi)(x_0), \exp^{-1}_{x_0}(h))_g \right)\mu(x_0,dh)\\
&=(b(x_0),\grad \phi(x_0))_g+\int_{\partial\Om}\left( \eta_{x_0}(h)\phi(h)-\Indicator_{B_{r_0}(x_0)}(h)(\grad \phi(x_0), \exp^{-1}_{x_0}(h))_g \right)\mu(x_0,dh)\\
&=(b(x_0), v)_g+\int_{B_{r_0+d_0^{\alpha_1}}(x_0)\setminus B_{r_0}(x_0)}\left(\eta_{x_0}(h)-1\right)(v, \exp^{-1}_{x_0}(h))_g\mu(x_0,dh).
\end{align*}
Now for points $x$, $y\in \partial \Om$ such that $y\in B_{r_0}(x)$ let $P_{x\to y}$ denote parallel transport of a tangent vector from $x$ to $y$ along the minimal geodesic connecting $x$ to $y$. In a manner similar to the construction of $\phi$, we take $\psi$ to be a $C^{2}$ function on $\partial \Om$ such that 
\begin{align*}
\psi(h)&=(P_{x_0\to y_0}v, \exp^{-1}_{y_0}(h))
\end{align*}
for $h\in B_{r_0}(y_0)$.

Then a similar calculation as above yields 
\begin{align*}
&\I(\eta_{y_0}\psi,y_0) = (b(y_0), P_{x_0\to y_0}v)_g+\int_{B_{r_0+d_0^{\alpha_1}}(y_0)\setminus B_{r_0}(y_0)}\eta_{y_0}(h)(P_{x_0\to y_0}v, \exp^{-1}_{y_0}(h))_g\mu(y_0,dh).
\end{align*}
Thus, using the fact that parallel transport preserves inner product,
\begin{align*}
 \abs{(b(x_0)-P_{y_0\to x_0} b(y_0), v)_g}&=\abs{(b(x_0), v)_g-(b(y_0), P_{x_0\to y_0}v)_g}.
 \end{align*}
Thus, using the triangle inequality, we can continue the previous as:
 \begin{align}
	\abs{(b(x_0)-P_{y_0\to x_0} b(y_0), v)_g} 
	&\leq\abs{\I(\eta_{x_0}\phi,x_0) -\I(\eta_{y_0}\psi,y_0)}\\
 &+\lvert\int_{B_{r_0+d_0^{\alpha_1}}(x_0)\setminus B_{r_0}(x_0)}\left(\eta_{x_0}(h)-1\right)(v, \exp^{-1}_{x_0}(h))_g\mu(x_0,dh)\rvert\notag\\
 &+\lvert\int_{B_{r_0+d_0^{\alpha_1}}(y_0)\setminus B_{r_0}(y_0)}\left(\eta_{y_0}(h)-1\right)(P_{x_0\to y_0}v, \exp^{-1}_{y_0}(h))_g\mu(y_0,dh)\rvert\notag\\
 &=: I+II+III.\label{eqn:forholderboundthis}
 \end{align}
 Now for the terms $II$ and $III$, we calculate using Theorem \ref{thm:MainLinear} part (ii),
\begin{align}
 II&\leq 2\int_{B_{r_0+d_0^{\alpha_1}}(x_0)\setminus B_{r_0}(x_0)}\abs{\exp^{-1}_{x_0}(h)}_g\mu(x_0,dh)\notag\\
 &\leq 2\Lambda\diam_g(\partial \Om)\int_{B_{r_0+d_0^{\alpha_1}}(x_0)\setminus B_{r_0}(x_0)}d(x_0, h)^{-n-1}\sigma(dh)\notag\\
 &\leq 2\Lambda\diam_g(\partial \Om) r_0^{-n-1}\sigma(B_{r_0+d_0^{\alpha_1}}(x_0)\setminus B_{r_0}(x_0))\leq Cd_0^{\alpha_1}\label{eqn:boundonII}
\end{align}
 for some universal $C>0$. We obtain the estimate for $III$ in the same.
 
 The remainder of the proof is to estimate the term $I$. Since $\eta_{x_0}\phi$ and $\eta_{y_0}\psi$ are $C^{2}$ functions on $\partial \Om$, we can use the results mentioned in the discussion preceding and following \eqref{eqDtoN:C1GammaEstimate}.  That is, there is some $\beta\in (0, 1)$, $0<\beta'<\beta$, and a universal $C>0$, so that
\begin{align}
I &\leq \abs{\partial_\nu U_{\eta_{x_0}\phi}(x_0)-\partial_\nu U_{\eta_{y_0}\psi}(x_0)}+\abs{\partial_\nu U_{\eta_{y_0}\psi}(x_0)-\partial_\nu U_{\eta_{y_0}\psi}(y_0)}\notag\\
 &\leq \abs{\I(\eta_{x_0}\phi, x_0)-\I(\eta_{y_0}\psi, x_0)}+C\abs{x_0-y_0}^{\beta'}[\nabla U_{\eta_{y_0}\psi}]_{C^{\beta'}(\Om)}\notag\\
 &\leq 
 \abs{\I(\eta_{x_0}\phi, x_0)-\I(\eta_{y_0}\psi, x_0)}+Cd_0^{\beta'}\norm{\eta_{y_0}\psi}_{C^{1,\beta}(\partial\Om)}.\label{eqn:boundonI}
\end{align}
It is easy to see that $\norm{\eta_{y_0}\psi}_{C^{1,\beta}(\partial\Om)}$ is bounded by a universal constant times $d_0^{-2\alpha_1}$, hence 
\begin{align}\label{eqn:boundonI.2}
 d_0^{\beta'}\norm{\eta_{y_0}\psi}_{C^{1,\beta}(\partial\Om)}\leq Cd_0^{\beta'-2\alpha_1}.
\end{align}
To deal with the first term, take $\rho>0$ much smaller than $r_0$ also to be determined later, and let $\tilde{\eta}\in C^\infty(\real)$ with $0\leq \tilde\eta\leq 1$, $\tilde{\eta}\equiv 0$ on $[-\rho,  \rho]$ and $\tilde{\eta}\equiv 1$ on $[2\rho, \infty)$, and define $\tilde{\phi}:=\tilde{\eta}(d(x_0, \cdot))\in C^{2, \beta}(\partial \Om)$. It is easy to see that both $\norm{\tilde\phi}_{C^2(\partial \Om)}$ and $\norm{1-\tilde\phi}_{C^2(\partial \Om)}$ are bounded by a universal constant times $\rho^{-2}$. We then apply \cite[Lemma 4.15 (4.7)]{GuSc-2016MinMaxNonlocalarXiv} and use Lemma \ref{lemWP:DtoNLipschitz} to see that (after possibly making a smaller choice for $\beta$),
\begin{align}
 \abs{\I(\eta_{x_0}\phi, x_0)-\I(\eta_{y_0}\psi, x_0)}&\leq C(\norm{(1-\tilde{\phi})(\eta_{x_0}\phi-\eta_{y_0}\psi)}_{C^{1, \beta}(\partial \Om)}+\norm{\tilde\phi}_{C^{1, \beta}(\partial \Om)}\norm{\eta_{x_0}\phi-\eta_{y_0}\psi}_{L^\infty(\spt (\tilde{\phi}))})\notag\\
 &\leq C(\norm{(1-\tilde{\phi})(\phi-\psi)}_{C^{1, \beta}(\partial \Om)}+\rho^{-1-\beta}\norm{\eta_{x_0}\phi-\eta_{y_0}\psi}_{L^\infty(\partial \Om\setminus B_{2\rho}(x_0))}).\label{eqn:DtoNboundwithcutoff}
\end{align}
Now note for any $h\in\partial \Om$,
\begin{align}
 \abs{\eta_{x_0}(h)\phi(h)-\eta_{y_0}(h)\psi(h)}
 &\leq  \abs{\eta_{x_0}(h)}\abs{\phi(h)-\psi(h)}+ \abs{\eta_{x_0}(h)\psi(h)-\eta_{y_0}(h)\psi(h)}\notag\\
 &\leq \abs{\eta_{x_0}(h)}\abs{\phi(h)-\psi(h)}+\diam_g(\partial \Om)\abs{\eta(d(x_0, h))-\eta(d(y_0, h))}\notag\\
 &\leq \abs{\eta_{x_0}(h)}\abs{\phi(h)-\psi(h)}+C\norm{\eta}_{C^1(\real)}\abs{d(x_0, h)-d(y_0,h)}\notag\\
 &\leq \abs{\eta_{x_0}(h)}\abs{\phi(h)-\psi(h)}+Cd_0^{1-\al_1}.\label{eqn:C^0differencebound}
\end{align}
The first term in the last line above is zero unless $d(x_0, h)\leq r_0+d_0^{\al_1}$. For such $h$ we find
\begin{align*}
 \abs{\eta_{x_0}(h)\phi(h)-\eta_{x_0}(h)\psi(h)}&\leq \abs{(v, \exp^{-1}_{x_0}(h)-P_{y_0\to x_0}\exp^{-1}_{y_0}(h))_g}\\
 &\leq \abs{\exp^{-1}_{x_0}(h)-P_{y_0\to x_0}\exp^{-1}_{y_0}(h)}_g\\
 &=\frac{1}{2}\abs{\left(\grad_x d(x, h)^2\vert_{x=x_0}-P_{y_0\to x_0}\grad_xd(x, h)^2\vert_{x=y_0}\right)}_g\\
 &\leq \frac{1}{2}\sup_{x\in \partial \Om, h\in B_{r_0}(x)}\abs{\Hess_x d(\cdot, h)^2}_g d(x_0, y_0)\leq Cd_0,
\end{align*}
where to obtain the third line above we have used \cite[Theorem 5.6.1 (5.6.4)]{Jost2011}. Thus if we take
\begin{align}\label{eqLin:TheJunProofRhoPowerOfD0} 
	\rho:=d_0^{\alpha_3/(1+\beta)}\ \text{for}\ \alpha_3\in (0, 1)\ \text{to be determined},
\end{align}
	by \eqref{eqn:C^0differencebound} we have
\begin{align}\label{eqn:firsttermbound}
 \rho^{-1-\beta}\norm{\eta_{x_0}\phi-\eta_{y_0}\psi}_{L^\infty(\partial \Om\setminus B_{2\rho(x_0)})}\leq Cd_0^{1-\al_1-\alpha_3}.
\end{align}
Next we turn to the term $\norm{(1-\tilde{\phi})(\phi-\psi)}_{C^{1, \beta}(B_{2\rho}(x_0))}$. First,
\begin{align*}
\norm{(1-\tilde{\phi})(\phi-\psi)}_{C^0(B_{2\rho}(x_0))}\leq \norm{\phi-\psi}_{C^0(B_{2\rho}(x_0))}\leq Cd_0
\end{align*}
by the same argument as above.

Next fix any $h\in B_{2\rho}(x_0)$, $w\in T_h(\partial \Om)$ and define for $t\in [0, 1]$ and $s$ near zero,
\begin{align*}
 \overline\gamma(s, t):&=\exp_{y_0}(t(\exp^{-1}_{y_0}(h)+s[D(\exp^{-1}_{y_0})\vert_{h}(w)])),\\
 J(t):&=\frac{\partial}{\partial s}\vert_{s=0}\overline\gamma(s, t),
\end{align*}
then $J$ is a Jacobi field along the geodesic from $y_0$ to $h$ with $J(0)=0$ and $\dot J=D(\exp_{y_0}^{-1})|_h(w)$ (see \cite[Sec 6.1.4]{Petersen-2016RiemannianBookGTM}). Then we calculate two different ways,
\begin{align}
 \frac{\partial}{\partial s}\vert_{s=0}\psi(\overline\gamma(s, 1))&=\frac{\partial}{\partial s}\vert_{s=0}(P_{x_0\to y_0}v,\exp^{-1}_{y_0}(h)+s[D(\exp^{-1}_{y_0})\vert_{h}(w)])_g\notag\\
 &=(P_{x_0\to y_0}v,D(\exp^{-1}_{y_0})\vert_{h}(w))_g=(v,P_{y_0\to x_0}[D(\exp^{-1}_{y_0})\vert_{h}(w)])_g,\notag\\
 \frac{\partial}{\partial s}\vert_{s=0}\psi(\overline\gamma(s, 1))&=(\grad \psi(h), J(1))_g=(\grad \psi(h), D(\exp_{y_0})\vert_{\exp^{-1}_{y_0}(h)}(\dot J(0)))_g\notag\\
 &=(\grad \psi(h), w)_g.\label{eqn:gradient formula}
\end{align}
Similarly,
\begin{align*}
 (v,D(\exp^{-1}_{x_0})\vert_{h}(w))_g &=(\grad \phi(h), w)_g.
\end{align*}
Thus for any $h_1$, $h_2\in B_{2\rho}(x_0)$ we have 
\begin{align}\label{eqn:crossdifference1}
 &(\grad\phi(h_1)-\grad\psi(h_1)-P_{h_2\to h_1}(\grad \phi(h_2)-\grad \psi(h_2)),w)_g\notag\\
 &=(v,D(\exp^{-1}_{x_0})\vert_{h_1}(w)-D(\exp^{-1}_{x_0})\vert_{h_2}(P_{h_1\to h_2}w)\notag\\
 &+(P_{y_0\to x_0}[D(\exp^{-1}_{y_0})\vert_{h_2}(P_{h_1\to h_2}w)]-P_{y_0\to x_0}[D(\exp^{-1}_{y_0})\vert_{h_1}(w)]))_g.
\end{align}
Let (all parametrized over $[0, 1]$) $\gamma$ and $h$ be the constant speed geodesics from $y_0$ to $x_0$, and $h_2$ to $h_1$ respectively, and $V$ and $W$ the parallel fields along $\gamma$ and $h$ respectively with $V(1)=v$ and $W(1)=w$. %, and define

Then the last expression in \eqref{eqn:crossdifference1} above can be written
\begin{align*}
&\int_0^1\int_0^1\frac{\partial }{\partial q}\frac{\partial }{\partial p}(V(p), [D(\exp^{-1}_{\gamma(p)})\vert_{h(q)}]W(q))_gdqdp\\
&=\int_0^1\int_0^1\frac{\partial }{\partial q}(V(p), \nabla_{\dot{\gamma}(p)}[D(\exp^{-1}_{\gamma(p)})\vert_{h(q)}]W(q))_gdqdp.
\end{align*}
Fix any local coordinates near $\gamma(p)$ and $h(q)$, then we find (below, all expressions are evaluated at $(x, h)=(\gamma(p), h(q))$)
\begin{align*}
& \nabla_{\dot{\gamma}(p)}[D(\exp^{-1}_{\gamma(p)})\vert_{h(q)}]W(q)=-\frac{1}{2}\nabla_{\dot{\gamma}(p)}[D_h\grad_x d(x, h)^2]W(q)\\
 &=-\frac{1}{2}\nabla_{\dot{\gamma}(p)}(g^{jk}(x)\partial^2_{x_kh_i}d^2(x, h)W^i(q)\partial_{x_j})\\
 &=-\frac{1}{2}\dot{\gamma}^l(p)[\partial_{x_l}(g^{jk}(x)\partial^2_{x_kh_i}d^2(x, h))+(g^{rk}(x)\partial^2_{x_kh_i}d^2(x, h))\Gamma^j_{lr}(x)]W^i(q)\partial_{x_j}
\end{align*}
where here, $\Gamma^i_{jk}$ are the Christoffel symbols.
 In particular $\nabla_{\dot{\gamma}(p)}[D(\exp^{-1}_{\gamma(p)})\vert_{h(q)}]W(q)$ is linear in $W(q)$, hence we can continue calculating as
\begin{align*}
 &\int_0^1\int_0^1\frac{\partial }{\partial q}(V(p), \nabla_{\dot{\gamma}(p)}[D(\exp^{-1}_{\gamma(p)})\vert_{h(q)}]W(q))_gdqdp\\
 &=\int_0^1\int_0^1\abs{\dot{\gamma}(p)}\frac{\partial }{\partial q}((\nabla_{\frac{\dot{\gamma}(p)}{\abs{\dot{\gamma}(p)}}}[D(\exp^{-1}_{\gamma(p)})\vert_{h(q)}])^tV(p), W(q))_gdqdp\\
 &=d(x_0, y_0)\int_0^1\int_0^1\abs{\dot{h}(q)}(\nabla_{\frac{\dot{h}(q)}{\abs{\dot{h}(q)}}}(\nabla_{\frac{\dot{\gamma}(p)}{\abs{\dot{\gamma}(p)}}}[D(\exp^{-1}_{\gamma(p)})\vert_{h(q)}])^tV(p), W(q))_gdqdp\\
 &\leq \sup_{p, q}\norm{(\nabla_{\frac{\dot{h}(q)}{\abs{\dot{h}(q)}}}(\nabla_{\frac{\dot{\gamma}(p)}{\abs{\dot{\gamma}(p)}}}[D(\exp^{-1}_{\gamma(p)})\vert_{h(q)}])^t} d(h_1, h_2)d_0\leq Cd(h_1, h_2)d_0
\end{align*}
for some constant $C>0$ depending only on $\partial \Om$ and $\norm{\cdot}$ is the operator norm above (again calculating in local coordinates shows $(\nabla_{\frac{\dot{\gamma}(p)}{\abs{\dot{\gamma}(p)}}}[D(\exp^{-1}_{\gamma(p)})\vert_{h(q)}])^t$ is a linear operator). Thus recalling \eqref{eqn:crossdifference1} we have
\begin{align}\label{eqn:holdergradientbound}
 \left[\grad(\phi-\psi)\right]_{C^{0, 1}(B_{2\rho}(x_0))}\leq Cd_0.
\end{align}

\noindent
Then for any $h\in B_{2\rho}(x_0)$ (recall $\rho$ from (\ref{eqLin:TheJunProofRhoPowerOfD0}))
\begin{align}
\abs{\grad(\phi-\psi)(h)}_g&\leq \abs{\grad(\phi-\psi)(x_0)}_g+Cd(x_0, h)d_0\notag\\
 &\leq \abs{v-\grad\psi(x_0)}_g+C\rho d_0\notag\\
 &=\abs{v-\grad\psi(x_0)}_g+Cd_0^{1+\alpha_3/(1+\beta)}.\label{eqn:C^1bound1}
\end{align}

\noindent
By \eqref{eqn:gradient formula} again we calculate for an arbitrary unit length $w\in T_{x_0}(\partial \Om)$,
\begin{align*}
\abs{ (v-\grad \psi(x_0), w)_g}&=\abs{(v, w-P_{y_0\to x_0}[D(\exp^{-1}_{y_0})\vert_{x_0}(w)])_g}\\
&\leq \abs{v}_g\abs{w-P_{y_0\to x_0}[D(\exp^{-1}_{y_0})\vert_{x_0}(w)]}_g\\
&\leq Cd_0,
\end{align*}
for some universal $C>0$. In particular, this gives $\abs{v-\grad\psi(x_0)}_g\leq Cd_0$, 
which combining with \eqref{eqn:C^1bound1} yields
\begin{align}
\norm{\grad(\phi-\psi)}_{C^0(B_{2\rho}(x_0))}\leq Cd_0.\label{eqn:C^1bound2}
\end{align}

Thus combining the above with \eqref{eqn:C^1bound1} we have 
\begin{align*}
\abs{\grad[(1-\tilde{\phi})(\phi-\psi)](h)}_g&\leq \abs{\grad(1-\tilde{\phi})}_g\abs{\phi(h)-\psi(h)}_g+\abs{\grad(\phi-\psi)(h)}_g\notag\\
&\leq C(\frac{d_0}{\rho}+d_0+d_0^{1+\frac{\alpha_3}{(1+\beta)}})\leq Cd_0^{1-\frac{\alpha_3}{(1+\beta)}}.
\end{align*}
Finally,

\begin{align*}
 \left[\grad((1-\tilde{\phi})(\phi-\psi))\right]_{C^{\beta}}
 &\leq  \norm{1-\tilde{\phi}}_{L^\infty}\left[\grad(\phi-\psi)\right]_{C^\beta}+\norm{\phi-\psi}_{L^\infty}\left[\grad(1-\tilde{\phi})\right]_{C^\beta}\\
 &\qquad+\norm{\grad(1-\tilde{\phi})}_{L^\infty}\left[\phi-\psi\right]_{C^\beta}+\norm{\grad(\phi-\psi)}_{L^\infty}\left[1-\tilde{\phi}\right]_{C^\beta}\\
 &\leq C(\frac{\rho d_0}{\rho^\beta}+\frac{d_0}{\rho^{1+\beta}}+\frac{d_0}{\rho^\beta})
 \leq C(\frac{d_0}{\rho^{1+\beta}})= Cd_0^{1-\alpha_3}
\end{align*}
where here all of the norms are taken over $B_{2\rho}(x_0)$. Thus we have shown that
\begin{align*}
 \norm{(1-\tilde{\phi})(\phi-\psi)}_{C^{1, \beta}(B_{2\rho}(x_0))}\leq Cd_0^{1-\alpha_3}.
\end{align*}
Now choose $\alpha_1$, $\beta$, and $\alpha_3\in (0, 1]$ so that $\alpha:=\min\{\alpha_1, \beta-2\alpha_1, 1-\alpha_3\}>0$, combining the final estimate above with \eqref{eqn:boundonI}, \eqref{eqn:boundonI.2}, \eqref{eqn:DtoNboundwithcutoff}, \eqref{eqn:firsttermbound} yields
\begin{align*}
 I\leq Cd_0^\alpha.
\end{align*}

Finally recalling \eqref{eqn:forholderboundthis}, \eqref{eqn:boundonII}, we will have for some universal $C>0$ and $\alpha\in (0, 1)$ the estimate
\begin{align*}
 \abs{(b(x_0)-P_{y_0\to x_0} b(y_0), v)_g}&\leq Cd(x_0, y_0)^\alpha
\end{align*}
which in turn proves that $b$ is locally H\"older continuous.
\end{proof}

%%%%%%%%%%%%%%%%%%%%%%%%%%%%%%%%%%%%%%%%%%%%%%
%%%%%%%%%%%%%%%%%%%%%%%%%%%%%%%%%%%%%%%%%%%%%%
%%%%%%%%%%%%%%%%%%%%%%%%%%%%%%%%%%%%%%%%%%%%%%
%%%%%%%%%%%%%%%%%%%%%%%%%%%%%%%%%%%%%%%%%%%%%%
%%%%%%%%%%%%%%%%%%%%%%%%%%%%%%%%%%%%%%%%%%%%%%

\subsection{The proof of Theorem \ref{thm:HolderXDependenceTVNorm}}
Here we provide the proof of the control of the H\"older continuity of the L\'evy measure with respect to the TV norm.

\begin{proof}[Proof of Theorem \ref{thm:HolderXDependenceTVNorm}]
Fix $\delta>0$, some $x_0\in \partial \Om$, and $r=\frac{\delta}{4}$. We assume that $2\delta<\min\{1,{\rm inj}(\partial\Om)\}$ where ${\rm inj}(\partial\Om)$ is the injectivity radius of $\partial \Om$. First we claim there exists $\alpha\in (0, 1)$ and $C>0$ such that if $\phi \equiv 0$ in $B_{2r}(x_0) \cap \partial \Omega$, then
\begin{align}\label{eq:localization of Holder regularity}
  \norm{\mathcal{I}(\phi,\cdot)}_{C^\alpha(B_{r}(x_0))} \leq \frac{C}{r}\norm{\phi}_{L^\infty(\partial \Omega)}.
\end{align}
Indeed, the claim immediately follows by the comparison principle combined with \cite[Corollary 8.36]{GiTr-98} in the divergence form case \eqref{eqIn:BulkFDiv}, and in the non-divergence form case \eqref{eqIn:BulkFNondiv}, it follows from \cite[Theorem 9.31 and eq (9.71)]{GiTr-98}

Now by \eqref{eqn:TVcharacterization} and density of $C^{1, \alpha}(\partial \Om)$ in $L^\infty(\partial \Om)$, it is sufficient to prove that for any $\phi\in C^{1, \alpha}(\partial \Om)$ with $\norm{\phi}_{L^\infty(\partial \Om)}\leq 1$,
   \begin{align}\label{eqn:TVprovethis}
    \abs{\int_{\partial \Omega} \phi(y) \chi_{\partial \Omega \setminus B_{\delta}(x_1)}(y)\mu(x_1,dy)-\int_{\partial \Omega} \phi(y) \chi_{\partial \Omega \setminus B_{\delta}(x_2) }(y) \mu(x_2,dy)} \leq Cd(x_1,x_2)^\alpha  
  \end{align}
  for some $C>0$ independent of $\alpha$, whenever $x_1$, $x_2\in B_r(x_0)$.
  
   Let $\eta_{k, x_1} \in C^2(\partial \Om)$ be such that $0\leq \eta_{k,x_1}\leq 1$ on $\partial \Om$, with $\eta_{k,x_1}\equiv 0$ on $B_{\delta}(x_1)$ and $\eta_{k,x_1}\equiv 1$ on $\partial \Om\setminus B_{\delta+1/k}(x_1)$, and an analogous choice for $\eta_{k, x_2}$. Then we find
\begin{align}
 \abs{\I(\eta_{k, x_1} \phi, x_1)-\I(\eta_{k, x_2} \phi, x_2)}&\leq \abs{\I(\eta_{k, x_1} \phi, x_1)-\I(\eta_{k, x_1} \phi, x_2)}+\abs{\I(\phi(\eta_{k, x_1} -\eta_{k, x_2}), x_2)}\notag\\
 &\leq \frac{C}{r}\norm{\phi}_{L^\infty(\partial \Omega)}d(x_1, x_2)^\alpha+\abs{\I(\phi(\eta_{k, x_1} -\eta_{k, x_2}), x_2)}\notag\\
 &\leq \frac{C}{r}d(x_1, x_2)^\alpha+\abs{\I(\phi(\eta_{k, x_1} -\eta_{k, x_2}), x_2)}\label{eqn:TVcrossterm}
\end{align}
where to obtain the second line we have used \eqref{eq:localization of Holder regularity} and that $x_1$, $x_2\in B_{r}(x_0)$, along with the choice of $r$; note that by the triangle inequality we have $\eta_{k, x_1}\equiv 0$ on $B_{2r}(x_0)$.
   
To estimate the second term in \eqref{eqn:TVcrossterm}, first we note by definition, $\eta_{k, x_1}-\eta_{k, x_2}=0-0= 0$ in $B_{\delta-d(x_1, x_2)}(x_2)$. Likewise, we have $\eta_{k, x_1}-\eta_{k, x_2}=1-1=0$ outside of $B_{\delta+d(x_1, x_2)+1/k}(x_2)$. Then by Theorem \ref{thm:MainLinear} (ii), we obtain
\begin{align*}
 \abs{\I(\phi(\eta_{k, x_1} -\eta_{k, x_2}), x_2)}&=\abs{\int_{B_{\delta+d(x_1, x_2)+1/k}(x_2)\setminus B_{\delta-d(x_1, x_2)}(x_2))}\phi(h)(\eta_{k, x_1}(h)-\eta_{k, x_2}(h))\mu(x_2, dh)}\\
 &\leq 2\Lambda\int_{B_{\delta+d(x_1, x_2)+1/k}(x_2)\setminus B_{\delta-d(x_1, x_2)}(x_2))}d(x_2, h)^{-n-1}\sigma(dh).
\end{align*}
Now we can consider normal coordinates centered at $x_2$, then writing $s$ for the radial coordinate and $\omega$ for coordinates on the unit sphere $\mathbb{S}^{n-1}$ we can write $\sigma=\lambda(s, \omega)ds\wedge \vol_{\mathbb{S}^{n-1}}$ for some real valued function $\lambda$ where $\vol_{\mathbb{S}^{n-1}}$ is the canonical volume form on $\mathbb{S}^{n-1}$. Since $\partial \Om$ is compact, there is a (possibly negative) lower bound $K$ on the Ricci curvature, thus using standard volume form comparison (see \cite[Lemma 7.1.2]{Petersen-2016RiemannianBookGTM}) we can calculate that
\begin{align*}
 \lambda(s, \omega)\leq {\rm sn}_{K}^{n-1}(s)\leq s^{n-1}+\frac{n-1}{2}\left(\max_{[0, {\rm inj}(\partial \Om)]}\abs{\ddot{{\rm sn}}_K}\right)s^n=s^{n-1}+Cs^n.
\end{align*}
Here 
\begin{align*}
{\rm sn}_K(s)=
\begin{cases}
 \frac{\sin{(s\sqrt{K})}}{\sqrt{K}},& K>0,\\
 s,& K=0,\\
  \frac{\sinh{(s\sqrt{-K})}}{\sqrt{-K}},& K<0,
\end{cases}
\end{align*}
and thus  $C>0$ only depends on $K$, $n$, and the injectivity radius ${\rm inj}(\partial \Om)$ of $\partial \Om$. Then we compute
\begin{align*}
 &\int_{B_{\delta+d(x_1, x_2)+1/k}(x_2)\setminus B_{\delta-d(x_1, x_2)}(x_2))}d(x_2, h)^{-n-1}\sigma(dh)\\
 &= \int_{\delta-d(x_1, x_2)}^{\delta+d(x_1, x_2)+1/k}\left(\int_{\mathbb{S}^{n-1}}s^{-n-1}\lambda(s, \omega)\vol_{\mathbb{S}^{n-1}}(d\omega)\right)ds\\
 &\leq \int_{\mathbb{S}^{n-1}}\vol_{\mathbb{S}^{n-1}}(d\omega)\int_{\delta-d(x_1, x_2)}^{\delta+d(x_1, x_2)+1/k}(s^{-2}+Cs^{-1})ds\\
 &\leq C\int_{\delta-d(x_1, x_2)}^{\delta+d(x_1, x_2)+1/k}s^{-2}ds=C\left(\frac{1}{\delta-d(x_1, x_2)}-\frac{1}{\delta+d(x_1, x_2)+\frac{1}{k}}\right)
\end{align*}
possibly taking $\delta$ smaller. Combining this with \eqref{eqn:TVcrossterm}, then taking $k\to \infty$ and using dominated convergence yields
\begin{align*}
 &\abs{\int_{\partial \Omega} \phi(y) \chi_{\partial \Omega \setminus B_{\delta}(x_1)}(y)\mu(x_1,dy)-\int_{\partial \Omega} \phi(y) \chi_{\partial \Omega \setminus B_{\delta}(x_2) }(y) \mu(x_2,dy)}\\
 &\leq C\left(\frac{1}{\delta}d(x_1, x_2)^\alpha+\frac{1}{\delta-d(x_1, x_2)}-\frac{1}{\delta+d(x_1, x_2)}\right).
\end{align*}
Finally,
\begin{align*}
 \frac{1}{\delta-d(x_1, x_2)}-\frac{1}{\delta+d(x_1, x_2)}&=\frac{2d(x_1, x_2)}{\delta^2-d(x_1, x_2)^2}\leq \frac{8d(x_1, x_2)}{3\delta^2}
\end{align*}
since $d(x_1, x_2)\leq 2r=\delta/2$, hence we obtain \eqref{eqn:TVprovethis}, finishing the proof.
\end{proof}

%%%%%%%%%%%%%%%%%%%%%%%%%%%%%%%%%%%%%%%%%%%%%%
%%%%%%%%%%%%%%%%%%%%%%%%%%%%%%%%%%%%%%%%%%%%%%
%%%%%%%%%%%%%%%%%%%%%%%%%%%%%%%%%%%%%%%%%%%%%%
%%%%%%%%%%%%%%%%%%%%%%%%%%%%%%%%%%%%%%%%%%%%%%
%%%%%%%%%%%%%%%%%%%%%%%%%%%%%%%%%%%%%%%%%%%%%%

%%%%%%%%%%%%%%%%%%%%%%%%%%%%%%%%%%%%%%%%%%%%%%
%%%%%%%%%%%%%%%%%%%%%%%%%%%%%%%%%%%%%%%%%%%%%%
%%%%%%%%%%%%%%%%%%%%%%%%%%%%%%%%%%%%%%%%%%%%%%
%%%%%%%%%%%%%%%%%%%%%%%%%%%%%%%%%%%%%%%%%%%%%%
%%%%%%%%%%%%%%%%%%%%%%%%%%%%%%%%%%%%%%%%%%%%%%

\section{Fully nonlinear equations-- Proof of Theorem \ref{thm:MainNonlinear}}\label{sec:FullyNonlinear}

In this section we treat fully nonlinear equations for (\ref{eqIn:BulkExtensionGeneric}) and (\ref{eqIn:BulkFNonlinear}), and we provide the proof of Theorem \ref{thm:MainNonlinear}.  We will collect some notation from Section \ref{sec:Intro}.  Recall, $\I$ is defined in (\ref{eqIn:BulkExtensionGeneric}) and (\ref{eqIn:DefOfDtoN}) under the nonlinear $F$ in (\ref{eqIn:BulkFNonlinear}).  Furthermore, Theorem \ref{thm:MainNonlinear} will show that for $\phi\in C^{1,\al}(\partial\Om)$,

\begin{align*}
	\I(\phi,x) = \min_{i}\max_{j}\left\{ f^{ij}(x) + L^{ij}(\phi,x)\right\},
\end{align*}
where $f^{ij}\in C(\partial\Om)$ and $L^{ij}$ are the linear operators defined as
\begin{align}\label{eqFN:LinearFamilyInMinMax}
		L^{ij}(\phi,x) &= c^{ij}(x)\phi(x) + \left(b^{ij}(x),\grad\phi(x)\right)\nonumber\\
		&+ \int_{\partial\Om} \left( \phi(h)-\phi(x)-
		\Indicator_{B_{r_0}(x)}(h)(\grad \phi(x), \exp^{-1}_x(h))_g \right)
		\mu^{ij}(x,dh). 
\end{align}

%%%%%%%%%%%%%%%%%%%%%%%%%%%%%%%%%%%%%%%%%%%%%%
%%%%%%%%%%%%%%%%%%%%%%%%%%%%%%%%%%%%%%%%%%%%%%
%%%%%%%%%%%%%%%%%%%%%%%%%%%%%%%%%%%%%%%%%%%%%%
%%%%%%%%%%%%%%%%%%%%%%%%%%%%%%%%%%%%%%%%%%%%%%
%%%%%%%%%%%%%%%%%%%%%%%%%%%%%%%%%%%%%%%%%%%%%%

\subsection{Proof of Theorem \ref{thm:MainNonlinear}, equation (\ref{eqIn:NonlinearMinMax})}

Thanks to the Lipschitz nature of $\I:C^{1,\al}(\partial\Om)\to C^\al(\partial\Om)$ that was established in Lemma \ref{lemWP:DtoNLipschitz}, the min-max formula promised in Theorem \ref{thm:MainNonlinear} is a consequence of \cite[Theorem 1.6 and Prop 1.7]{GuSc-2016MinMaxNonlocalarXiv} (see also \cite[Theorem 1.8]{GuSc-2016MinMaxNonlocalarXiv} which even establishes that $L^{ij}$ are linear operators mapping $C^{1,\al}(\partial\Om)\to C^\al(\partial\Om)$).  Now we focus on the more specific behavior of $\mu^{ij}$ and $b^{ij}$.

%%%%%%%%%%%%%%%%%%%%%%%%%%%%%%%%%%%%%%%%%%%%%%
%%%%%%%%%%%%%%%%%%%%%%%%%%%%%%%%%%%%%%%%%%%%%%
%%%%%%%%%%%%%%%%%%%%%%%%%%%%%%%%%%%%%%%%%%%%%%
%%%%%%%%%%%%%%%%%%%%%%%%%%%%%%%%%%%%%%%%%%%%%%
%%%%%%%%%%%%%%%%%%%%%%%%%%%%%%%%%%%%%%%%%%%%%%
\subsection{Reduction to the extremal operators}

A very useful tool for obtaining the estimates (i-a) and (i-b) in Theorem \ref{thm:MainNonlinear} is the reduction from a general $F$ in (\ref{eqIn:BulkFNonlinear}) to the particular instance of the Pucci operator, $F=\M^-$.  This is a consequence of the representation of the extremal operators of $\I$ in terms of the D-to-N for $\M^-$, which appeared in Lemma \ref{lemWP:DtoNExtremals}.  Specifically, we record the result of \cite[Prop 4.35]{GuSc-2016MinMaxNonlocalarXiv} as it pertains to $\I$ in this work.  As the proof of this proposition is not particular to the D-to-N mapping, we refer to \cite[Sec 4.6]{GuSc-2016MinMaxNonlocalarXiv} for its proof.

\begin{prop}[see Proposition 4.35, Sec 4.6 of \cite{GuSc-2016MinMaxNonlocalarXiv}]\label{propFN:ExtremalLowerBoundLevyOperator}
	If $L^{ij}$ is any one of the collection of linear operators appearing in Theorem \ref{thm:MainNonlinear}, defined in (\ref{eqFN:LinearFamilyInMinMax}), then for all $\phi\in C^3_c(\partial\Om)$, the following estimate holds:
	\begin{align*}
		M^-(\phi,x)\leq L^{ij}(\phi,x)\leq M^+(\phi,x).
	\end{align*}
	Here, $M^\pm$ are the extremal operators defined in Lemma \ref{lemWP:DtoNExtremals}.  
\end{prop}

Proposition \ref{propFN:ExtremalLowerBoundLevyOperator} means that in order to establish the estimates in Theorem \ref{thm:MainNonlinear}, we can focus on obtaining, e.g. lower bounds for $M^\pm(\phi,x)$. This is a welcome simplification to the problem, for example because $\M^\pm$ (for equation (\ref{eqIn:BulkExtensionGeneric})) are convex/concave as well as rotation and translation invariant, and they enjoy good regularity theory ($C^{2,\al}$ boundary data produces $C^{2,\al'}$ solutions).

%%%%%%%%%%%%%%%%%%%%%%%%%%%%%%%%%%%%%%%%%%%%%%
%%%%%%%%%%%%%%%%%%%%%%%%%%%%%%%%%%%%%%%%%%%%%%
%%%%%%%%%%%%%%%%%%%%%%%%%%%%%%%%%%%%%%%%%%%%%%
%%%%%%%%%%%%%%%%%%%%%%%%%%%%%%%%%%%%%%%%%%%%%%
%%%%%%%%%%%%%%%%%%%%%%%%%%%%%%%%%%%%%%%%%%%%%%
\subsection{The ring estimate, Theorem \ref{thm:MainNonlinear} (i-a)}

Here we provide the proof of the ring estimate that appears in Theorem \ref{thm:MainNonlinear} (i-a).

\begin{proof}[Proof of Theorem \ref{thm:MainNonlinear} part (i-a)]
	
	First, we note that $x\in\partial\Om$ is just a parameter, and a translation of the equation (\ref{eqIn:BulkExtensionGeneric}) so that $x=0$ does not change any of the assumptions on $F$.  Thus, without loss of generality, we take $x=0\in\partial\Om$.  We will obtain the desired ring estimate by rescaling the domain in (\ref{eqIn:BulkExtensionGeneric}) from $\Om$ to a larger set, $(1/r)\Om$, and representing $U_\phi$ in $\Om$ as a rescaling of an appropriate function, $\tilde U_{\tilde \phi}$, in $(1/r)\Om$.  The advantage here is to utilize the fact that $\displaystyle\partial \left((1/r)\Om\right)$ is becoming flat in a $C^2$ fashion under this scaling, and so we can use solutions in one fixed domain to build appropriate sub and super solutions for equations in $(1/r)\Om$.  We now proceed with the construction.
	
	Thanks to Lemma \ref{lemTOOL:RingsAmbientIntrinsic}, we will work with functions and sets in $\real^{n+1}$ and actually show a related estimate  (which is no harm when $r$ is small).  When $B^{n+1}_r\subset \real^{n+1}$ is the usual ball in $\real^{n+1}$, we will prove:
	\begin{align}\label{eqFN:RingNewGoalLower}
		 C_1 r^{-1}\leq\mu^{ij}(x, (B^{n+1}_{(7/4)r}\setminus B^{n+1}_{(5/4)r})\intersect\partial\Om)
	\end{align}
	and
	\begin{align}\label{eqFN:RingNewGoalUpper}
		\mu^{ij}(x, (B^{n+1}_{(9/4)r}\setminus B^{n+1}_{(3/4)r})\intersect\partial\Om)
		\leq C_2 r^{-1}.
	\end{align}
	Thus, for ease of presentation let us introduce the notation for respectively the small and big rings:
	\begin{align*}
		R^S_r := (B^{n+1}_{(7/4)r}\setminus B^{n+1}_{(5/4)r})\intersect\partial\Om
		\ \ \text{and}\ \ 
		R^B_r := (B^{n+1}_{(9/4)r}\setminus B^{n+1}_{(3/4)r})\intersect\partial\Om.
	\end{align*}
	The reason for this simplification is to be able to work with $\phi$ that are actually defined in all of $\real^{n+1}$, and use their restrictions to various submanifolds as Dirichlet data.  To this end, let $\phi^r_l$ and $\phi^r_u$ be $C^2(\real^{n+1})$ lower and upper barrier functions such that
   	\begin{align}\label{eqFN:PhiRDef1Inequality}
   		0\leq \phi^r_l\leq \Indicator_{R^S_r}\leq \Indicator_{R^B_r}\leq \phi^r_u,
   	\end{align}
   	and furthermore, just for concreteness, we assume $\phi_r^l$, $\phi_r^u$ satisfy
   	\begin{align}\label{eqFN:PhiRDef2lower}
   		\begin{cases}
   		\phi^r_l&\equiv 1\ \text{in}\ B^{n+1}_{(13/8)r}\setminus B^{n+1}_{(11/8)r}\\
   		\phi^r_l&\equiv 0\ \text{outside}\ B^{n+1}_{(14/8)r}\setminus B^{n+1}_{(10/8)r}
   		\end{cases}
   	\end{align}
   	and
   	\begin{align}\label{eqFN:PhiRDef3upper}
   		\begin{cases}
   		\phi^r_u&\equiv 1\ \text{in}\ R_r^B\\
   		\phi^r_u&\equiv 0\ \text{outside}\ B^{n+1}_{(19/8)r}\setminus B^{n+1}_{(5/8)r}.
   		\end{cases}
   	\end{align}
   	Thus we see that
   	\begin{align}\label{eqFN:RingMassAndPhiLowerIntegrals}
   		\int_{\partial\Om\setminus\{x\}}\phi^r_l(y)\mu^{ij}(x,dy)
   		\leq \int_{\partial\Om\setminus\{x\}}\Indicator_{R^S_r}(y)\mu^{ij}(x,dy)
   	\end{align}
	and
   	\begin{align}\label{eqFN:RingMassAndPhiUpperIntegrals}
		\leq \int_{\partial\Om\setminus\{x\}}\Indicator_{R^B_r}(y)\mu^{ij}(x,dy)
   		\leq \int_{\partial\Om\setminus\{x\}}\phi^r_u(y)\mu^{ij}(x,dy).
   	\end{align}
   	Furthermore, since
   	\begin{align*}
   		\phi^r_l(x)=\phi^r_u(x)=0\ \text{and}\ \grad\phi^r_l(x)=\grad\phi^r_u(x)=0,
   	\end{align*}
   	we see that the operators in (\ref{eqFN:LinearFamilyInMinMax}) simplify to
   	\begin{align}\label{eqFN:RingMassGoalAsLOnPhiLowerUpper}
   		L^{ij}(\phi^r_l,x) = \int_{\partial\Om} \phi^r_l(y)\mu^{ij}(x,dy)\ \ \ 
   		\text{and}\ \ \  
   		L^{ij}(\phi^r_u,x) = \int_{\partial\Om} \phi^r_u(y)\mu^{ij}(x,dy).
   	\end{align}
   	Thus, to conclude (\ref{eqFN:RingNewGoalLower}) and (\ref{eqFN:RingNewGoalUpper}), it suffices, via Proposition \ref{propFN:ExtremalLowerBoundLevyOperator} combined with (\ref{eqFN:RingMassGoalAsLOnPhiLowerUpper}), (\ref{eqFN:RingMassAndPhiLowerIntegrals}), and (\ref{eqFN:RingMassAndPhiUpperIntegrals}) to show the same bounds for the normal derivatives of the functions $U_{\phi^r_l}$ and $U_{\phi^r_u}$ that solve (\ref{eqIn:BulkExtensionGeneric}) with respectively $F=\M^-$ and $F=\M^+$.

	Now, we record our target to achieve (\ref{eqFN:RingNewGoalLower}) and (\ref{eqFN:RingNewGoalUpper}).  Assume that $U_{\phi^l_r}$ and $U_{\phi^u_r}$ are respectively the solutions of (\ref{eqIn:BulkExtensionGeneric}) for $F=\M^-$ and $F=\M^+$ with Dirichlet data given respectively by $\displaystyle\phi^r_l|_{\partial\Om}$ and $\displaystyle\phi^r_u|_{\partial\Om}$.  We will show
\begin{align}\label{eqFN:RingGoalNormalDerivativeUsualScale}
\text{\textbf{goal:} there are universal constants so that}\ 
C_1r^{-1}\leq\partial_\nu U_{\phi^r_l}(0)\ \ 
\text{and}\ \ 
\partial_\nu U_{\phi^r_u}(0)\leq C_2 r^{-1}.
\end{align}
We will give the details for the bound on $\partial_\nu U_{\phi^r_l}$, and the upper bound for $\partial_\nu U_{\phi^r_u}$ will follow analogously.

We will represent $U_{\phi^r_l}$ as a rescaling of a particular function, $\tilde U$, in a larger domain, by defining
\begin{align*}
	\tilde \Om_r = (1/r)\Om,
\end{align*}
\begin{align*}
	\M^-(\tilde U) = 0\ \text{in}\ \tilde\Om_r\ \text{and}\ 
	\tilde U|_{\partial\tilde\Om_r} = \phi^l_1|_{\partial\tilde\Om_r},
\end{align*}
and
\begin{align*}
	U_{\phi^l_r}(y) = \tilde U(\frac{y}{r})\ \text{for}\ y\in\Om.
\end{align*}
This means that
\begin{align*}
	\displaystyle M^-(\phi^l_1|_{\partial\Om},y) = \partial_\nu U_{\phi^l_r}(y)
	= r^{-1}\partial_\nu \tilde U(y).
\end{align*}
Thus, our new goal will be to show that
\begin{align}\label{eqFN:UnscaledGoal}
	\textbf{unscaled goal:}\ C_1\leq \partial_\nu \tilde U(0).
\end{align}

In order to get a lower estimate on $\partial_\nu \tilde U$ that is truly independent of $r$, we will use an auxiliary function that is independent of $r$ and defined in a fixed domain, independent of $r$.  Let us call the ``half'' ball,
\begin{align*}
	B^+_{10}(0) := \{y\in\real^{n+1}\ :\ y\cdot\nu(0)>0\}\intersect B^{n+1}_{10}.
\end{align*}
Then we can define the function $\tilde V$ as the unique solution of
\begin{align*}
	\M^-(\tilde V)=0\ \text{in}\ B^+_{10},\ \ \text{and}\ \ 
	\tilde V|_{\partial B^+_{10}} = \phi^l_1|_{\partial B^+_{10}}.
\end{align*}
The advantage of $\tilde V$ is that it is independent of $r$, and so as long as we can show that $\tilde U-\tilde V$ is small enough in the $C^{1,\al}$ sense, then we will be able to conclude the auxiliary goal in (\ref{eqFN:UnscaledGoal}).

In order to get the estimate between $\tilde U$ and $\tilde V$, we must introduce two more auxiliary functions.  The first is $\tilde W_r$, defined in the domain $\tilde B$,
\begin{align*}
	\tilde B := \tilde\Om_r\intersect B^+_{10},
\end{align*}
and 
\begin{align*}
	\M^-(\tilde W_r)=0\ \ \text{in}\ \ \tilde B,\ \ \text{and}\ \ 
	\tilde W_r|_{\partial\tilde B} = \phi^l_1|_{\partial\tilde B}.
\end{align*}
Thus, since $\tilde U>0$ inside $\Om_r$, we see that $\tilde W_r$ is a subsolution (including ordering of boundary data) to the equation for $\tilde U$ (or vice-versa, $\tilde U$ is a supersolution for the equation for $\tilde W_r$), hence by the comparison principle,
\begin{align*}
	\tilde U\geq \tilde W_r\ \text{in}\ \tilde B,\ \ \text{and}\ \ 
	\partial_\nu \tilde U(0)\geq \partial_\nu \tilde W_r(0).
\end{align*}
Now, to conclude, we will show a lower bound for $\partial_\nu\tilde W_r(0)$.

We note that the distance between the half space determined by the tangent to $\partial\tilde\Om_r$ in $B^{n+1}_{10}$, $\{y\in\real^{n+1}\ :\ y\cdot\nu(0)>0\}\intersect B^{n+1}_{10}$, and to $\partial\tilde\Om_r\intersect B^{n+1}_{10}$ is vanishing as $r\to0$ (in particular, it is of order $Cr$).  Furthermore, by the boundary estimates in \cite[Theorem 1.1]{SilvestreSirakov-2013boundary}, we know that 
\begin{align*}
	\norm{\tilde U}_{C^{1,\al}(\tilde\Om_r\intersect B^{n+1}_{10})}\leq C\norm{\phi^l_1|_{\partial\tilde\Om_r}}_{C^{1,\al}}\leq C,
\end{align*}
(note, by the Evans-Krylov Theorem, $\tilde U$ is actually $C^{2,\gam}$, but we only invoke estimates for $\tilde U$ and $\grad \tilde U$).
Hence, by the flattening of $\partial\tilde\Om_r$ as $r\to0$, we see that on the ``lower'' boundary of $B^+_{10}$, we can make $\tilde U$ and $\phi^l_1$ close:
\begin{align*}
	\norm{\tilde U-\phi^l_1}_{C^{1,\al}(\{y\in\real^{n+1}\ :\ y\cdot\nu(0)=0\}\intersect B^{n+1}_{10})}
	\to0\ \ \text{as}\ \ r\to0.
\end{align*}
This means that if we define the function
\begin{align*}
	\tilde Z = \tilde W_r-\tilde V,
\end{align*}
then in the common domain of their equations, we have, by the properties of viscosity solutions
\begin{align*}
	\M^-(\tilde Z)\leq 0\ \ \text{and}\ \ \M^+(\tilde Z)\geq 0\ \ \text{in}\ \ \tilde\Om_r\intersect B^+_{10},
\end{align*}
and thanks to the boundary values for $\tilde V$
\begin{align*}
	\norm{\tilde Z|_{\partial (\tilde\Om_r\intersect B^+_{10})}}_{C^{1,\al}}\to 0\ \text{as}\ r\to0.
\end{align*}
Furthermore, since $\tilde V$ is independent of $r$, and since $\tilde V$ attains a minimum at $y=0\in\partial B^+_{10}$ by the Hopf principle, we know that for a $C$ that depends only on universal parameters and the choice of $\phi^l_1$,
\begin{align*}
	\partial_\nu \tilde V(0) = C>0.
\end{align*}

Hence, taking $R$ small enough (recall $R$ from Theorem \ref{thm:MainNonlinear} (i-a)), so that for $r\leq R$, we have
\begin{align*}
	\norm{\tilde Z|_{\partial (\tilde\Om_r\intersect B^+_{10})}}_{C^{1,\al}}< \frac{C}{2},
\end{align*}
we can then conclude for these $r\leq R$ that
\begin{align*}
	\partial_\nu \tilde W_r(0)\geq \frac{C}{2},
\end{align*}
and also, by the above comparison of $\tilde W_r$ and $\tilde U$,
\begin{align*}
	\partial_\nu \tilde U(0)\geq \frac{C}{2}=C_1,
\end{align*}
where $C_1$ is a universal constant.
As noted earlier, rescaling $\tilde U$, gives the lower bound.

The proof of the upper bound follows analogously.  Instead of using $\M^-$ to define the functions $\tilde U$, $\tilde W_r$, $\tilde V$, $\tilde Z$, we will use the operator $\M^+$.  Also, at the stage of using comparison to switch from $\tilde U$ to $\tilde W_r$, it will be useful to use boundary data that is identically $1$ outside of $B^{n+1}_{3/4}$ so that $\tilde W_r$ can serve as a supersolution for $\tilde U$.  Thus, this same function will be used to determine the boundary values of $\tilde V$, instead of $\phi^u_1$, which would have been the direct analog of the argument.  Everything else follows similarly.  
\end{proof}

%%%%%%%%%%%%%%%%%%%%%%%%%%%%%%%%%%%%%%%%%%%%%%
%%%%%%%%%%%%%%%%%%%%%%%%%%%%%%%%%%%%%%%%%%%%%%
%%%%%%%%%%%%%%%%%%%%%%%%%%%%%%%%%%%%%%%%%%%%%%
%%%%%%%%%%%%%%%%%%%%%%%%%%%%%%%%%%%%%%%%%%%%%%
%%%%%%%%%%%%%%%%%%%%%%%%%%%%%%%%%%%%%%%%%%%%%%

\subsection{A lower bound for $\mu^{ab}(x,\cdot)$ in Theorem \ref{thm:MainNonlinear} part (i)(b)}

Next, we prove the lower bound for $\mu^{ab}$ in Theorem \ref{thm:MainNonlinear} (i-b). Our approach will be to work in the context of linear equations with smooth coefficients, and invoke some techniques and results about the related Green's functions from e.g. \cite{Kenig-1993PotentialThoeryNonDiv}.  In order to transfer results between fully nonlinear equations and equations with smooth coefficients, we will collect various facts and observations from the literature.    This first fact is a technique for approximating solutions of fully nonlinear equations by those of linear equations with smooth coefficients.  It is more or less well known to specialists, but there does not seem to be any standard reference.  Here we present the technique as used by Feldman \cite[Proof of Prop. 2.2]{Feldman-2001LipschitzFB2Phase}, where it is proved in complete detail.  Since this is nearly exactly as implemented in \cite{Feldman-2001LipschitzFB2Phase}, we simply list a sketch of the steps without detailed justification/explanation.

\begin{lem}[Smooth Linear Approximation]\label{lemFN:ApproxByLinEqs}
	Any solution of Pucci's equation can be approximated by solutions of linear equations with smooth coefficients and the same ellipticity bounds.
	
	Given $\phi\in C(\partial\Om)$ and $U_\phi$ solving (\ref{eqIn:BulkExtensionGeneric}) and (\ref{eqIn:BulkFNonlinear}) with $F(D^2U,x)=\M^-(D^2U)$, there exists a family of coefficients, $A^\del(x)$, depending on $U_\phi$, which are uniformly elliptic all with the same constants $(\lam,\Lam)$ and smooth in $x$, such that for $U^\del_\phi$ solving 
	\begin{align*}
		\begin{cases}
		\Tr(A^\del(x)D^2U^\del_\phi)=0,\ &\text{in}\ \Om\\
		U^\del_\phi=\phi\ &\text{on}\ \partial\Om,
		\end{cases}
	\end{align*}
	we recover
	\begin{align*}
		\norm{U^\del_\phi-U_\phi}_{L^\infty(\Om)}\to0\ \ \text{as}\ \del\to0.
	\end{align*}
\end{lem}

\begin{proof}[Proof of Lemma \ref{lemFN:ApproxByLinEqs}]
	Again, as mentioned above, we present only a sketch of the proof that comes from \cite[Prop. 2.2]{Feldman-2001LipschitzFB2Phase}.

	Here are the steps:
	\begin{enumerate}
		\item Approximate $\M^-$ by smooth concave functions, $\M^{-,k}$, giving $w^k$ that solve the smoothed equation.  For eventual limiting operations via the stability of viscosity solutions, this requires that $\M^{-,k}\to\M^-$ uniformly on compact subsets of $\mathcal{S}((n+1)\times(n+1))$
		\item Linearize $\M^{-,k}$ over $w^k$, and use the fact that $w^k$ are $C^{2,\al}(\Om)$ (see e.g. \cite{CaCa-95}), which can be done explicitly as
		\begin{align*}
			a^k_{i,j}(x):=\int_0^1 \frac{\partial\M^{-,k}}{\partial P_{i,j}}(sD^2w^k(x))ds.
		\end{align*}
		
		\item Extend $a^k_{i,j}$ to all of $\real^{d+1}$ as simply $a^k_{i,j}(x)=\del_{i,j}$ for all $x\not\in\Om$ (note $\del$  is the Kronecker delta symbol).
		
		\item Mollify $a^k_{i,j}$ to be smooth, denoting them as $a^{k,m}_{i,j}$.
		
		\item Taking the matrix $A^{k,m}=(a^{k,m}_{i,j})$, solve the equation
		\begin{align*}
			\begin{cases}
				\Tr(A^{k,m}D^2w^{k,m})=0\ &\text{in}\ \Om\\
				w^{k,m}=\phi\ &\text{on}\ \partial\Om.
			\end{cases}
		\end{align*}

		\item Confirm that there exists a subsequence $w^k\to U_\phi$ uniformly in $\overline{\Om}$ as $k\to\infty$, as well as a subsequence $w^{k,m}\to w^k$ uniformly in $\overline{\Om}$ for $k$ fixed and $m\to\infty$.  In both cases, one can invoke, for example $C^\al$ estimates, as all of these functions are uniformly bounded with a common bound.  The first convergence and stability result uses regular viscosity solutions theory, and the second convergence uses the the $L^p$ viscosity solutions theory in e.g. \cite{CaCrKoSw-96}.  We note that the limit in both cases uses the fact that viscosity solutions are stable and that the limit equations have unique solutions.
	\end{enumerate}
	
	We briefly remark that the reason for invoking the $L^p$ theory is that it is not known how good are the coefficients $a^k_{i,j}$ in the vicinity of $\partial\Om$.  It seems reasonable in this lemma to want to keep the same boundary values throughout the whole process.  We note that if it so happens that $\phi\in C^{2,\al}(\partial\Om)$, then one can use regular viscosity solutions for both convergence arguments, as this would produce $w^k\in C^{2,\al}(\overline{\Om})$, and hence $a^k_{i,j}\in C^\al(\overline{\Om})$.
\end{proof}

This next lemma is a simple exercise for constructing a sequence of balls linking points in $\Om$, each of whose radius is a (fixed) multiple of the previous.  For $C^2$ domains, it is a simpler property than the Harnack chains that are used in \cite{JerisonKenig-1982BoundaryBehaviourNTADomADViM}, but we keep the same name nonetheless.  We omit the proof. 

\begin{lem}[The Harnack chain distance]\label{lemFN:HarnackChain}  
	If $\partial\Om$ is bounded and $C^2$, then there is a universal $R_0$ so that if $r>0$ is fixed, and  $y\in\Om$ and $x\in\Om\intersect B_{R_0}(y)$, with $d(y,\partial\Om)>2r$ and $d(x,\partial\Om)>2r$ then $x$ and $y$ can be linked by a Harnack chain based on balls of multiples of radius, $r$, so that 
	$N=\#\{\text{balls in the chain}\}\leq C_1\log(\frac{C_2\abs{x-y}}{r})$.  Here, the constants $C_1$ and $C_2$ are independent from $r$, and they depend only on $n$, $\lam$, $\Lam$.
	
	(We note that by Harnack chain based on balls of radius $r$, we mean a sequence of balls that successively overlap, twice each is contained in $\Om$, the first contains $y$ and the last contains $x$, and all of their radii are multiples of $r$.)
\end{lem}

In the next couple of results, in investigating the Harmonic measure of $B_r^{n+1}(h)\intersect\Om$, it will be useful to use an auxiliary ball that is actually inside $\Om$, and has both a size and distance to $\partial\Om$ that are comparable to $r$.  We call this ball, $\tilde B_r$, and we record its definition here:
\begin{DEF}\label{defFN:AuxiliaryBall}
	Given $h\in\partial\Om$ and $B_r^{n+1}(h)\intersect\Om$, the auxiliary ball is $\tilde B_r=B_r^{n+1}(h+4r\cdot \nu (h))$ (recall $\nu(h)$ is the inward normal vector at $h$).  We will call $\hat h=h+4r\cdot \nu (h)$
\end{DEF}

The next two results apply to any operator of the form $L_Au(x)=\Tr(A(x)D^2u(x))$ such that $A$ is smooth and uniformly $\lam,\Lam$-elliptic.  The resulting bounds depend only on dimension and $\lam,\Lam$.  However, we only use them for the $A^\del$ produced by Lemma \ref{lemFN:ApproxByLinEqs}, and so that is how we will present their results.  The next two lemmas are a blending of ideas from \cite[Section 5]{Kenig-1993PotentialThoeryNonDiv} and \cite[Appendix B]{CaSoWa-05}.

\begin{lem}\label{lemFN:GreenMassOnABall}
	Let $A^\del$ be as in Lemma \ref{lemFN:ApproxByLinEqs} and $G^\del$ be the Green's function for $A^\del$ in $\Om$.
	If $x\in\Om$, $h\in\partial\Om$, $\abs{x-h}=l$, $r$ is small enough, $d(x,\partial\Om)>2r$, and $\tilde B_r$ is the ball $B_r^{n+1}(h+4r\nu(h))\subset\Om$, then there exists a universal $\eta\geq n$ and $C$ so that
	\begin{align*}
		C (\frac{r}{l})^{\eta}\leq \frac{1}{r^2}\int_{\tilde B_r} G^\del(x,z)dz.
	\end{align*}
\end{lem}

\begin{lem}\label{lemFN:HarmonicMeasureGreenIntegralEstimate}
	Let $A^\del$ be as in Lemma \ref{lemFN:ApproxByLinEqs}, $G^\del$ the Green's function for $A^\del$ in $\Om$, and $\om^\del$ be the harmonic measure for $A^\del$ in $\Om$.
	If $x\not\in B_r^{n+1}(h+4\nu(h))$, and $r$ is small enough,
	\begin{align*}
		\om_x(B^{n+1}_r(h)\intersect\partial\Om)
		\geq \frac{1}{r^2}\int_{\tilde B_{r}}G^\del(x,z)dy,
	\end{align*}
	where as above we are using $\tilde B_r=B_r^{n+1}(h+4\nu(h))$.
\end{lem}

\begin{rem}
	We believe it may be worth noting that although the result claimed in Lemma \ref{lemFN:GreenMassOnABall}, especially if $x$ approaches $\partial\Om$, seems strange, there is no contradiction in the inequality.  Even though one expects $\int_{\tilde B_r}G^\del(x,z)dz\leq cr^2d(x,\partial\Om)$ (as will be apparent from the subsequent proofs, combined with boundary behavior), there is a restriction for the Harnack chain that $d(x,\partial\Om)\geq 2r$.  Thus, in the worst case, if we take $d(x,\partial\Om)=2r$, we see that Lemma \ref{lemFN:GreenMassOnABall} will imply $c(r/l)^\eta\leq d(x,\partial\Om)=2r$, and this inequality does not cause a problem. The usefulness of the inequality will be when $d(x,\partial\Om)$ is of order $l$, which is much larger than $r$.
\end{rem}

First, we will prove Lemma \ref{lemFN:GreenMassOnABall}.

\begin{proof}[Proof of Lemma \ref{lemFN:GreenMassOnABall}]
	Let $x$, $h$, and $r$ be fixed as in the statement of the lemma.
	Let us define the function 
	\begin{align}\label{eqFN:wFunctionDefinition}
		w(y)=\int_{\tilde B_r} G^\del(y,z)dz=\int_\Om\Indicator_{\tilde B_r}(z)G^\del(y,z)dz.
	\end{align}
	That is to say, that by definition, $w$ is the unique function that solves
	\begin{align}\label{eqFN:wEqForLowerBoundGreenIntegral}
		\begin{cases}
		L_{A^\del}w = -\Indicator_{\tilde B_r}\ &\text{in}\ \Om\\
		w=0\ &\text{on}\ \partial\Om.
		\end{cases}
	\end{align}
	First, using a comparison argument, we will get a lower bound on $w$ in $\tilde B_r$.  Then we will iterate it using a Harnack chain until we reach $x$.
	
	Let us recall $\hat h=h+4r\cdot \nu (h)$. We note that for an appropriate choice of $\theta$ (depending only on $n$, $\lam$, $\Lam$), the function $p(y)=\theta(r^2-\abs{y-\hat h}^2)$ satisfies 
	\[
	L_{A^\del}p \geq -\Indicator_{\tilde B_r}\ \text{in}\ \tilde B_r\ \ 
	\text{and}\ \ p=0\ \text{on}\ \partial\tilde B_r.
	\]
	Thus, by comparison, we have obtained that 
	\[
	w(\hat h)\geq \frac{3}{4}\theta r^2,\ \ \text{in}\ \frac{1}{2}\tilde B_r = B^{n+1}_{r/2}(\hat h).
	\]
	Now, using a barrier for $L_{a^\del}u=0$ in $B^{n+1}_{2r}(\hat h)\setminus B^{n+1}_{r/2}(\hat h) $, we can conclude that 
	\[
	w\geq c_0\theta r^2\ \ \text{in}\ \ B^{n+1}_{3r/2}\hat h.
	\]
	By iterating Harnack's inequality in a Harnack chain of balls proportional to $\tilde B_r$, we see that if $N$ is the number of such balls required to link $\hat h$ to $x$, then there is a universal $C>1$ (arising from the Harnack inequality) so that
	\[
	w(x)\geq \left(\frac{1}{C}\right)^N\theta r^2.
	\]
	Thus, invoking the Harnack chain bound in Lemma \ref{lemFN:HarnackChain}, we see that 
	\[
	w(x)\geq \left(\frac{1}{C}\right)^{C_1\log(\frac{C_2l}{r})}\theta r^2.
	\]
	By setting $\eta$ as,
	\[
	\eta= -\log\left( \left(\frac{1}{C}\right)^{C_1	}   \right),
	\]
	we see then that
	\[
	\left(\frac{1}{C}\right)^{C_1\log(\frac{C_2l}{r})} = \left( \frac{r}{C_2l}\right)^\eta.
	\]
	Hence, we see that for another, universal, $\tilde C$,
	\[
	w(x)\geq \tilde C \left(\frac{r}{l}\right)^\eta r^2.
	\]
	Dividing by $r^2$, relabeling $\tilde C$, and recalling (\ref{eqFN:wFunctionDefinition}) conclude the lemma.
\end{proof}

Before we give a proof of Lemma \ref{lemFN:HarmonicMeasureGreenIntegralEstimate}, we will need a result about a barrier function. We will use the fact that because $\Om$ has the uniform exterior ball condition, given a point, $h\in\partial\Om$, we can choose an annulus, for constants $c_0$ and $R$ that depend only on $\Om$, so that for an appropriate $y_0$, $\Om\subset B^{n+1}_R(y_0)\setminus B^{n+1}_{c_0}(y_0)$ and $B^{n+1}_{c_0}(y_0)$ is tangent to $\partial\Om$ at $h\in\partial\Om$.

\begin{lem}\label{lemFN:ExternalBallBarrier}
	Assume that  $c_0>0$ and $r>0$ are given, with $r<c_0$. There exists a function, $\psi$, that solves in the viscosity sense,
	\begin{align*}
		\M^+(\psi)&\leq -1\ \ \text{for}\ \ c_0-r<\abs{x}<c_0 + 5r,\\
		\M^+(\psi)&\leq 0\ \ \text{for}\ \ c_0-r<\abs{x},
	\end{align*}
	with
	\[
	\psi\geq0\ \text{in}\ \real^{n+1}\ \ \text{and}\ \ \sup_{\real^{n+1}}\left( \psi \right)\leq cr^2.
	\]
\end{lem} 

The proof of Lemma \ref{lemFN:ExternalBallBarrier} is an explicit calculation, and we defer its proof until after the proof of Lemma \ref{lemFN:HarmonicMeasureGreenIntegralEstimate}. 

\begin{proof}[Proof of Lemma \ref{lemFN:HarmonicMeasureGreenIntegralEstimate}]
	This proof appears, for example, in the proof of \cite[Lemma 5.18]{Kenig-1993PotentialThoeryNonDiv}.  We give slightly more detail here.  We recall that
	\begin{align*}%\label{eqFN:TildeBRecallLem5.6}
		\tilde B_r=B^{n+1}_r(h+4\nu(h)).
	\end{align*}

	Let $x$, $h$, and $r$ be fixed as in the statement of the lemma.
	Let us recall the function, $w$, as described in (\ref{eqFN:wFunctionDefinition}) and (\ref{eqFN:wEqForLowerBoundGreenIntegral}).
	
	For simplicity, for $y\in\Om$ let us just call $v(y)=\om_y(B^{n+1}_r(h)\intersect\partial\Om)$.  We know from the definition of harmonic measure that $v$ satisfies
	\begin{align*}
		\begin{cases}
			L_{A^\del}v = 0\ &\text{in}\ \Om\\
			v=\Indicator_{B^{n+1}_r(h)\intersect\Om}\ &\text{on}\ \partial\Om.
		\end{cases}
	\end{align*}
	
	We wish to establish the lemma by the following two claims, followed by the maximum principle in $\Om\setminus \tilde B_r$, as $v$ and $w$ satisfy the ordering
	\begin{align*}
		v|_{\partial\Om}\geq 0=w|_{\partial\Om}.
	\end{align*}
	
	\underline{Claim 1:} for some universal $c>0$, $w$ satisfies the estimate $\sup_{\tilde B_r}w\leq cr^2$.
	
	\underline{Claim 2:} for some universal $c>0$, $\inf_{\tilde B_r}v\geq c$.

	First, we address claim 1.  We use the exterior ball condition for $\Om$ with balls of radius, $c_0$.  Thus, there is some $\tilde h\not\in\Om$ so that $B_{c_0}^{n+1}(\tilde h)\subset \Om^C$ and  is tangent to $\partial\Om$ at $h$. After an appropriate translation, we see that the function, $\psi$, from Lemma \ref{lemFN:ExternalBallBarrier} can be made to be a super solution in the set, $\abs{y}>c_0-r$ (as $\M^+\psi\geq L_{A^\del}\psi$, by definition of $\M^+$), which contains $\Om$.   Furthermore, by construction, after a translation, we will have $\M^+\psi\leq -1$ in $\tilde B_r$.  Hence, this translation of $\psi$ is a super solution for the same equation as $w$, and that by construction, $\psi\geq 0$ on $\partial\Om$.  Hence claim 1 follows from the comparison theorem for $L_{A^\del}$ in $\Om$ and the estimate that $\sup(\psi)\leq cr^2$.
	
	To see why claim 2 is true, we invoke \cite[Lemma 5.3]{Kenig-1993PotentialThoeryNonDiv} (which also comes from  \cite{FabesGarofaloMarinMalaveSalsa-1988FatouTheoremsNonDivREVMATIBERO}), for the function $(1-v)$, first in $B_{r/2}(h)\intersect\Om$.  That is to say that since $\sup_{B_r\intersect\Om} (1-v)\leq 1$, we see that for some universal $\bar\al$
	\begin{align*}
		\text{for}\ x_1=h+(\frac{r}{2}\nu),\ \ (1-v(x_1))\leq \left(\frac{1}{2}\right)^{\bar\al}.
	\end{align*}
	In other words, $v(x_1)\geq 1- (1/2)^{\bar\al}$.  Now, for example, with
	\begin{align*}
		x_2 = h+\frac{5}{8}\nu\ \ \text{and}\ \ x_3=h+\frac{7}{8}\nu,
	\end{align*}
	using a ball of radius $3r/16$, we see that Harnack's inequality applies so that
	\begin{align*}
	 1-(\frac{1}{2})^{\bar\al}\leq v(x_1)\leq \sup_{B_{3r/16}(x_2)}v\leq C\inf_{B_{3r/16}(x_2)}\leq Cv(x_3)
	\end{align*}
	Repeating this process three more times (with slightly larger radii) gives that if $y\in\tilde B_r$, $\inf_{\tilde B_r}v\geq \frac{1}{C}$.

	Now, to conclude the theorem, we see that after multiplying by an appropriate universal constant, 
	\begin{align*}
		v(y)\geq \frac{C}{r^2}w(y)\ \ \text{for all}\ y\in\tilde B_r.
	\end{align*}
	Hence, by the above observation that $v$ and $w$ solve the same equation (with zero right hand side) in  $\Om\setminus \tilde B_r$, we conclude that 
	\begin{align*}
		v(y)\geq \frac{C}{r^2}w(y)\ \ \text{for all}\ y\in\tilde \Om,
	\end{align*}
	and this implies them lemma, taking $y=x$.
\end{proof}

Finally, we are in a position put the steps together to prove Theorem \ref{thm:MainNonlinear} part (i)(b).  

\begin{proof}[Proof of Theorem \ref{thm:MainNonlinear} part (i)(b)]
	We first assume that $x\in\partial\Om$, $h\in\partial\Om$, $r>0$ are fixed, that $x\not=h$, and $r<(d(x,h))/10$.  For this part of the proof, it is easiest to assume that $d(x,h)$ is small enough so that if $\abs{x-h}=l$, then $x+l\nu(x)\in\Om$ and $B_{2r}^{n+1}(x+l\nu(x))\subset\Om$. This is not a restriction, as we have already assumed that $\Om$ is bounded and $\partial\Om$ is $C^2$.  (We also note an intentional switch to using $\abs{x-h}$ in this section as we can assume this is comparable to $d(x,h)$.)

	We note that just as above, we shall assume that $\phi$ is smooth and 
	\begin{align*}
		\phi\geq \Indicator_{B_r^{n+1}(h)\intersect\Om}.
	\end{align*}
	The result will follow by taking a sequence of such $\phi$, decreasing to $\Indicator_{B_r(h)}$, but we suppress the sequence for now to keep the notation to a minimum.  The key properties of $\phi$ that we assume are 
	\[
	\phi(x)=0,\ \ \text{and}\ \ \grad\phi(x)=0.
	\]
	We now remind the reader that for this part of the theorem, if $\mu^{ij}$ are as in (\ref{eqFN:LinearFamilyInMinMax}) (which is given by the first part of the theorem), we must show that 
	\begin{align*}
		\mu^{ab}(B_r^{n+1}(h)\intersect\Om)\geq \frac{cr^\eta}{d(x,h)^{\eta+1}}.
	\end{align*}
	According to our choice of $\phi$, combined with the formula in (\ref{eqFN:LinearFamilyInMinMax}), and that $\grad\phi(x)=\phi(x)=0$,
	\begin{align*}
		L^{ij}(\phi,x)= \int_{\partial\Om} \phi(z)\mu^{ij}(x,dz).
	\end{align*}
	Thus, in other words, our goal can be recast as showing 
	\begin{align*}
		L^{ij}(\phi,x)\geq \frac{cr^\eta}{d(x,h)^{\eta+1}},
	\end{align*}
	and hence since the lower bound uses only that $\phi\geq\Indicator_{B_r^{n+1}(r)\intersect\Om}$, the claim will follow by letting $\phi$ decrease pointwise to $\Indicator_{B_r^{n+1}(r)\intersect\Om}$.  Again, as above, this lower bound can be obtained by finding a lower bound for the extremal operators, per Proposition \ref{propFN:ExtremalLowerBoundLevyOperator}.  Thus, Proposition \ref{propFN:ExtremalLowerBoundLevyOperator} shows the following estimate will suffice:
	\begin{align}\label{eqFN:MainThmExtremalLowerBoundGoal}
		M^-(\phi,x_0)\geq \frac{cr^\eta}{d(x,h)^{\eta+1}},
	\end{align}
	where we recall the D-to-N extremal operator, $M^-$, defined in \ref{eqWP:ExtremalForBoundaryOpDef}.
	
	Now, assume that $U_\phi$ is the unique solution of 
	\begin{align}\label{eqFN:MainProofExtremalEqPhi}
		\begin{cases}
			M^-(U_\phi,y)=0\ &\text{in}\ \Om\\
			U_\phi=\phi\ &\text{on}\ \partial\Om.
		\end{cases}
	\end{align}
	We will focus on the values of $U_\phi(\hat x)$, where $\hat x$ is chosen so that
	\begin{align*}
		\hat x = x + l\nu(x)\ \ \text{recall}\ (l=\abs{x-h}).
	\end{align*}
	Invoking barriers, such as those of the form $C(d(y,\partial\Om)+cd(y,\partial\Om)^2)$, which are subsolutions to (\ref{eqFN:MainProofExtremalEqPhi}), we see that if we can show that 
	\begin{align}\label{eqFN:MainThmULowerBoundGoal}
		U_\phi(\hat x)\geq C\left( \frac{r}{l} \right)^\eta,
	\end{align}
	then it follows, with linearly growing barriers, that
	\begin{align*}
		U_\phi(y)\geq C\frac{d(y,\partial\Om)\left( \frac{r}{l} \right)^\eta}{l}.
	\end{align*}
	Hence, as soon as we obtain (\ref{eqFN:MainThmULowerBoundGoal}), it follows that
	\begin{align*}
		\partial_\nu U_\phi(x)\geq \frac{Cr^\eta}{l^{\eta+1}},
	\end{align*}
	which is exactly what is needed, via $M^-(\phi,x)$ to obtain (\ref{eqFN:MainThmExtremalLowerBoundGoal}).

	Given that the goal in (\ref{eqFN:MainThmExtremalLowerBoundGoal}) is a pointwise bound, and given that we can approximate $U_\phi$ and (\ref{eqFN:MainProofExtremalEqPhi}) via solutions to linear equations to smooth coefficients using Lemma \ref{lemFN:ApproxByLinEqs}, it suffices to show that the $U^\del_\phi$ in Lemma \ref{lemFN:ApproxByLinEqs} also enjoys
	\begin{align*}
		U^\del_\phi(\hat x)\geq C\left( \frac{r}{l} \right)^\eta.
	\end{align*}
	However, this last equation follows immediately from Lemmas \ref{lemFN:GreenMassOnABall} and \ref{lemFN:HarmonicMeasureGreenIntegralEstimate}, combined with the fact that $\phi\geq \Indicator_{B_r^{n+1}(h)\intersect\Om}$ and using the comparison principle for the functions $U^\del_\phi(y)$ and $v(y)=\om_y(B^{n+1}_r(h)\intersect\partial\Om)$.
\end{proof}

\begin{rem}
	The reader should see, through the details of the proof, that in the nice case of linear equations with H\"older coefficients, we will recover $\eta=n$.  Indeed, in Lemma \ref{lemFN:GreenMassOnABall}, this result follows immediately from the estimates on Green's functions invoked in Section \ref{sec:LinearEq}.  However, in the absence of these estimates, the seemingly only available tool was Harnack's inequality, at which point multiple invocations of it will lead to some $\eta$ that is expected to be significantly larger than $n$ (this is in Lemma \ref{lemFN:GreenMassOnABall}).  This means that in the nonlinear setting, the L\'evy measures may assign a much smaller mass to balls than in the linear case with H\"older coefficients.
\end{rem}

To conclude this section, we will give the calculation that leads to the barrier in Lemma \ref{lemFN:ExternalBallBarrier}.  

\begin{prop}\label{propFN:RadialFunctionStartingPoint}
	Given any $b>0$, there exists $\ep_0>0$ and $a_0<1/2$ that are independent from $b$ and depend only on $\lam$, $\Lam$, $n$, such that there exists a function, $f$, that solves in the viscosity sense:
	\begin{align*}
		f&\geq0\ \ \text{in}\ \ \real^{n+1}\\
		\M^+(f)&\leq 0\ \ \text{for}\ \abs{x}>a_0b\\
		\intertext{and}
		\M^+(f)&\leq -\ep_0\ \ \text{for}\ \ a_0b<\abs{x}<\frac{b}{2}.
	\end{align*}
\end{prop}

\begin{proof}[proof of Proposition \ref{propFN:ExtremalLowerBoundLevyOperator}]
	We first begin with the function, $g$, defined as 
	\begin{align*}
		g(t)=
		\begin{cases}
			-t(t-b)\ &\text{if}\ t\in[0,\frac{b}{2})\\
			\frac{b^2}{4}\ &\text{if}\ t\in[\frac{b}{2},\infty).
		\end{cases}
	\end{align*}
	We will construct $f$ as 
	\[
	f(x)=g(\abs{x}).
	\]
	Thus, computing derivatives, we see that
	\begin{align*}
		\frac{\partial^2 f}{\partial x_i\partial x_j}(x)=
		\begin{cases}
			g''(\abs{x})\frac{x_ix_j}{\abs{x}^2} + 
			g'(\abs{x})\left(\frac{1}{\abs{x}}-\frac{x_i^2}{\abs{x}}\right)\ &\text{if}\ i=j
			\vspace{0.1in}\\
			g''(\abs{x})\frac{x_ix_j}{\abs{x}^3} 
			-g'(\abs{x})\frac{x_ix_j}{\abs{x}^3}\ &\text{if}\ i\not=j.
		\end{cases}
	\end{align*}
	Furthermore, since $f$ is a radial function and $\M^+$ is a rotationally invariant operator, it suffices to check the equation only for $\M^+(f,te_1)$.  First, we do this for the case of $t\in (0,\frac{b}{2})$.
	
	Plugging in $x=te_1$ to the second derivatives of $f$ shows
	\begin{align*}
		\frac{\partial^2 f}{\partial x_i\partial x_j}(te_1)=
		\begin{cases}
			g''(t)\ &\text{if}\ i=j=1\\
			g'(t)t^{-1}\ &\text{if}\ i=j\not=1\\
			0\ &\text{otherwise}.
		\end{cases}
	\end{align*}
	Thus, computing $\M^+(f,te_1)$ (recall $x\in\real^{n+1}$), we get
	\begin{align*}
		\M^+(f,te_1)&= n\Lam g'(t)t^{-1} + \lam g''(t)\\
		&= n\Lam(-2+\frac{b}{t}) - 2\lam,
	\end{align*}
	where we note we have used that $g'(t)\geq0$ when $t\in(0,\frac{b}{2})$. We now see that 
\begin{align*}
	\lim_{t\to(b/2)} g'(t)t^{-1}=0,\ \ \text{and hence}\ \ \lim_{t\to(b/2)}\M^+(f,te_1)=-2\lam.
\end{align*}
Thus, to be concrete, we may choose $\ep_0=\lam$, from which the existence of $a_0<\frac{1}{2}$ follows from the fact that $\M^+(f,te_1)$ is strictly decreasing for $t<\frac{b}{2}$ and sufficiently close to $\frac{b}{2}$.

The previous calculation verifies the claimed inequality for $\M^+(f,x)$ for $\abs{x}\in(a_0b,\frac{b}{2})$.  In order to confirm the remaining cases of $x$, we simply note that at all $x$ with $\abs{x}>a_0b$, we have that $f$ is either twice differentiable at $x$, or any test function, $\phi$, must satisfy $D^2\phi(x)\leq 0$ at any points where $f-\phi$ attains a minimum.  Hence we obtain the the equation for $\abs{x}\in(a_0b,\infty)$.  (We note to the reader that avoiding a neighborhood of $x=0$ is intentional, as $f$ can be touched from below by functions with a positive Hessian there.)  This concludes our proof.
\end{proof}

Now that we have the basic function, $f$, the proof of Lemma \ref{lemFN:ExternalBallBarrier} follows as a simple corollary.

\begin{proof}[Proof of Lemma \ref{lemFN:ExternalBallBarrier}]
	Starting with the function, $f$, and $b=12r$, from Proposition \ref{propFN:RadialFunctionStartingPoint}, the function $\psi$ can be constructed using suitable choices of a dilation, a shift, and a multiplication by a constant.  Furthermore, all of these operations depend upon and change $f$ by only factors that are universal in the sense of depending on the exterior ball radius, $c_0$, and $\lam$, $\Lam$, $n$.    Since, by construction, $f$ enjoys the bound, $\sup f\leq b^2/4$, we see that after these transformations, we will retain $\psi\leq cr^2$ for some universal $c$.
\end{proof}

%%%%%%%%%%%%%%%%%%%%%%%%%%%%%%%%%%%%%%%%%%%%%%
%%%%%%%%%%%%%%%%%%%%%%%%%%%%%%%%%%%%%%%%%%%%%%
%%%%%%%%%%%%%%%%%%%%%%%%%%%%%%%%%%%%%%%%%%%%%%
%%%%%%%%%%%%%%%%%%%%%%%%%%%%%%%%%%%%%%%%%%%%%%
%%%%%%%%%%%%%%%%%%%%%%%%%%%%%%%%%%%%%%%%%%%%%%

%%%%%%%%%%%%%%%%%%%%%%%%%%%%%%%%%%%%%%%%%%%%%%
%%%%%%%%%%%%%%%%%%%%%%%%%%%%%%%%%%%%%%%%%%%%%%
%%%%%%%%%%%%%%%%%%%%%%%%%%%%%%%%%%%%%%%%%%%%%%
%%%%%%%%%%%%%%%%%%%%%%%%%%%%%%%%%%%%%%%%%%%%%%
%%%%%%%%%%%%%%%%%%%%%%%%%%%%%%%%%%%%%%%%%%%%%%

\section{Comments on more general boundary conditions}

For elliptic equations, such as (\ref{eqIn:BulkExtensionGeneric}) with $F$  as in (\ref{eqIn:BulkFDiv})--(\ref{eqIn:BulkFNonlinear}), two of the most natural boundary conditions (depending upon whom is asked) would be $U|_{\partial\Om}=\phi$ and $\partial_\nu U=g$.  This paper, of course gives a description of the link between the two.  However, the Neumann condition, $\partial_\nu U=g$, is just the prototype of this family, and there are many other possibilities, such as oblique, capillarity, geometric, and Robin:
\begin{align*}
	&B(x)\cdot\grad U(x)=g(x),\ \ \text{with}\ \ B(x)\cdot\nu(x)\geq \lam>0\\
	&\partial_\nu U(x)=g(x)\sqrt{1+\abs{\grad U(x)}^2}\\
	&\partial_\nu U(x)=g(x)\abs{\grad U(x)}\\
	&\partial_\nu U(x)=g(x)U(x).
\end{align*} 
In all cases, these types of boundary conditions can be written generically as 
\begin{align*}
	G(x,U,\grad U)=0,
\end{align*}
where $G$ is increasing with respect to $\partial_\nu U$.  The requirement that $G$ is increasing comes from the fact that $G$ is used in conjunction with an elliptic equation, for which the comparison principle is essential, and hence the relevant $G$ all also enjoy this monotonicity property with respect to $\partial_\nu U$.  This means the standard assumption is that
\begin{align*}
	G(x,r,p+c\nu(x))-G(x,r,p)\geq \lam c\ \ (\text{or, more generally}, >0),
\end{align*}
combined with natural growth restrictions jointly in the $x,r,p$ variables.  There are many works on this topic, but we point to Barles \cite{Barles-1999NeumannQuasilinJDE}, Lieberman-Trudinger \cite{LiebermanTrudinger-1986NonlinearObliqueBCTAMS}, and \cite{LionsTrudinger-1986ObliqDerivHJBMathZ} for a sample of results and more references.

The key point about these more general Neumann-type operators, $G$, is that they all obey the global comparison property, and under natural ellipticity assumptions, it is not hard to check that they too, just as with $\I$, will be Lipschitz mappings of $C^{1,\al}(\partial\Om)\to  C^{\al}(\partial\Om)$.
What this means in the context of our operator, $\I$, is that many, if not all of the results of Theorems \ref{thm:MainLinear} -- \ref{thm:MainNonlinear} should have direct analogs to the case of the operator, $\G$, which is defined as
\begin{align*}
	\G(\phi,x) = G(x,\phi(x),\grad U_\phi(x)),
\end{align*}  
where $U_\phi$ is as in (\ref{eqIn:BulkExtensionGeneric}) and $G$ is as above.

%%%%%%%%%%%%%%%%%%%%%%%%%%%%%%%%%%%%%%%%%%%%%%
%%%%%%%%%%%%%%%%%%%%%%%%%%%%%%%%%%%%%%%%%%%%%%
%%%%%%%%%%%%%%%%%%%%%%%%%%%%%%%%%%%%%%%%%%%%%%
%%%%%%%%%%%%%%%%%%%%%%%%%%%%%%%%%%%%%%%%%%%%%%
%%%%%%%%%%%%%%%%%%%%%%%%%%%%%%%%%%%%%%%%%%%%%%
%%%%%%%%%%%%%%%%%%%%%%%%%%%%%%%%%%%%%%%%%%%%%%
%%%%%%%%%%%%%%%%%%%%%%%%%%%%%%%%%%%%%%%%%%%%%%
%%%%%%%%%%%%%%%%%%%%%%%%%%%%%%%%%%%%%%%%%%%%%%
%%%%%%%%%%%%%%%%%%%%%%%%%%%%%%%%%%%%%%%%%%%%%%
%%%%%%%%%%%%%%%%%%%%%%%%%%%%%%%%%%%%%%%%%%%%%%

\section{Open Questions}\label{sec:Open}

We believe there are at least a few natural open questions that arise as a result of Theorems \ref{thm:MainLinear}--\ref{thm:MainNonlinear}.  Here we briefly explain some of them.

\textbf{Lack of symmetry.}  Even in the case of a flat domain, $\Om=\real^{n+1}_+$, one does not expect the resulting integro-differential operators to have symmetric kernels in the sense that $k(x,x-h)=k(x,x+h)$ (or in the nonlinear case $\mu(x,B_r(x+h))=\mu(x,B_r(x-h))$).  In the linear case, one can see this immediately from the fact that if you set $x=0$, then, except in special circumstances, one will have $U_{\phi}\not= U_{\phi(-\cdot)}$ (or, more importantly, one would need equality with reflected data at all $x$).  This suggests that both a nonzero drift term, $b$, and the lack of symmetry of $\mu$ is inevitable.   Thus, going forward, it will be important to characterize this lack of symmetry in a precise way.  Furthermore, we suggest that this drift and lack of symmetry will also be present in the case of nonlinear equations, even for the D-to-N operators for the Pucci operators.  Does it correspond to any assumptions that are similar to those in \cite{Chan-2012NonlocalDriftArxiv}, \cite{ChDa-2012NonsymKernels}, or \cite{SchwabSilvestre-2014RegularityIntDiffVeryIrregKernelsAPDE}?  Or is it a different type altogether?

Furthermore, there will be another source that destroys the symmetry of the L\'evy measures from the curvature of $\partial\Om$ when $\partial\Om$ is not flat.  This effect also needs to be made precise.

\textbf{Regularity theory on manifolds.}  As mentioned in the introduction, one of the  attractive features of viewing some operators from an integro-differential viewpoint is the possibility to invoke regularity results that depend on minimal assumptions on the L\'evy measure (these are referred to as Krylov-Safonov type results in Section \ref{sec:Background}).  As it currently stands, to the best of our knowledge, there seem to be no such results when $\Om$ is anything other than $\real^{n+1}_+$ (meaning $\I$ acts on functions on $\real^n$).  It should be useful in the future to have analogs of the results mentioned in Section \ref{sec:Background}, appropriately modified to account for the correct assumptions that would be found by making the lack of symmetry precise.

\textbf{Regularity theory with different lower bounds on $\mu$.}  The property (i)-b of the integro-differential operators resulting from Theorem \ref{thm:MainNonlinear} is new for the existing literature.  Presumably regularity results are obtainable for such situations, but at the moment, it is a completely open question.

%%%%%%%%%%%%%%%%%%%%%%%%%%%%%%%%%%%%%%%%%%%%%%
%%%%%%%%%%%%%%%%%%%%%%%%%%%%%%%%%%%%%%%%%%%%%%
%%%%%%%%%%%%%%%%%%%%%%%%%%%%%%%%%%%%%%%%%%%%%%
%%%%%%%%%%%%%%%%%%%%%%%%%%%%%%%%%%%%%%%%%%%%%%
%%%%%%%%%%%%%%%%%%%%%%%%%%%%%%%%%%%%%%%%%%%%%%
\nocite{Courrege-1965formePrincipeMaximum}  
\bibliography{../refs}
\bibliographystyle{plain}
%%%%%%%%%%%%%%%%%%%%%%%%%%%%%%%%%%%%%%%%%%%%%%
\end{document}